\documentclass[oneside,reqno]{amsart}
\usepackage{amsmath,amssymb,amscd,latexsym,verbatim,multicol}
\usepackage{a4wide} 
\usepackage{graphicx} 
\usepackage{color} 
\usepackage{boxedminipage}
\usepackage{mathbbol}
\usepackage{nameref}

\usepackage{enumerate}
\usepackage[all]{xy}
\xyoption{curve}
\SelectTips{eu}{}
\CompileMatrices

  \newenvironment{note}[1][Note]
   {\bigskip\begin{center}\begin{boxedminipage}{4.5in}\setlength{\parindent}{1em}\noindent\textbf{#1. }}
   {\end{boxedminipage}\end{center}\bigskip}

\newcommand{\asym}{\operatorname{asym}}
\newcommand{\rad}{\mathrm{radical}}
\newcommand{\prim}{\mathrm{prime}}
\newcommand{\ci}{\mathrm{c.i.}}
\newcommand{\pmd}{\operatorname{pmd}}
\newcommand{\ww}{{\mathfrak{w}}}
\newcommand{\ee}{{\mathfrak{e}}}

\newcommand{\height}{{\operatorname{height}}}
\newcommand{\Sym}{{\operatorname{Sym}}}

\newcommand{\GL}{\operatorname{GL}}

\newcommand{\chara}{\operatorname{char}}

\newcommand{\Ker}{\operatorname{Ker}}

\newcommand{\rank}{\operatorname{rank}}

\newcommand{\generic}{\ensuremath{\mathrm{gen}}}
\newcommand{\symmetric}{\ensuremath{\mathrm{sym}}}
\newcommand{\skewsymmetric}{\ensuremath{\mathrm{skew}}}

\newcommand{\Tor}{\operatorname{Tor}}

\newcommand{\ini}{\operatorname{in}}

\newcommand{\Pf}{\operatorname{Pf}}
\newcommand{\Sp}{\operatorname{Sp}}
\newcommand{\posMa}{\operatorname{pm}}

\newcommand{\NN}{{\mathbb N}}
\newcommand{\QQ}{{\mathbb Q}}
\newcommand{\RR}{{\mathbb R}}
\newcommand{\ZZ}{{\mathbb Z}}

\newcommand{\KK}{{\mathbb{k}}}

\newcommand{\OO}{{\mathcal O}}

\newcommand{\pP}{{\mathcal P}}
\newcommand{\vV}{{\mathcal V}}
\newcommand{\group}{{\mathfrak{G}}}

  \makeatletter
  \CheckCommand*\refstepcounter[1]{\stepcounter{#1}%
      \protected@edef\@currentlabel
       {\csname p@#1\endcsname\csname the#1\endcsname}%
  }
  \renewcommand*\refstepcounter[1]{\stepcounter{#1}%
    \protected@edef\@currentlabel
      {\csname p@#1\expandafter\endcsname\csname the#1\endcsname}%
  }
  \def\labelformat#1{\expandafter\def\csname p@#1\endcsname##1}
  \DeclareRobustCommand\Ref[1]{\protected@edef\@tempa{\ref{#1}}%
     \expandafter\MakeUppercase\@tempa
  }
  \makeatother

  \makeatletter
  \newcommand{\numberlike}[2]{%
     \expandafter\def\csname c@#1\endcsname{%
         \expandafter\csname c@#2\endcsname}%
  }
  \makeatother

  \def\DefaultNumberTheoremWithin{section}

  \theoremstyle{plain}
  \newtheorem{Lemma}{Lemma}
     \numberwithin{Lemma}{\DefaultNumberTheoremWithin}
     \labelformat{Lemma}{Lemma~#1}
  
     \numberwithin{Claim}{\DefaultNumberTheoremWithin}
     \numberlike{Claim}{Lemma}
     \labelformat{Claim}{Claim~#1}
  \newtheorem{Theorem}{Theorem}
     \numberwithin{Theorem}{\DefaultNumberTheoremWithin}
     \numberlike{Theorem}{Lemma}
     \labelformat{Theorem}{Theorem~#1}
  \newtheorem{Corollary}{Corollary}
     \numberwithin{Corollary}{\DefaultNumberTheoremWithin}
     \numberlike{Corollary}{Lemma}
     \labelformat{Corollary}{Corollary~#1}
  \newtheorem{Proposition}{Proposition}
     \numberwithin{Proposition}{\DefaultNumberTheoremWithin}
     \numberlike{Proposition}{Lemma}
     \labelformat{Proposition}{Proposition~#1}
  \newtheorem{Conjecture}{Conjecture}
     \numberwithin{Conjecture}{\DefaultNumberTheoremWithin}
     \numberlike{Conjecture}{Lemma}
     \labelformat{Conjecture}{Conjecture~#1}

  \theoremstyle{definition}
  \newtheorem{Definition}{Definition}
     \numberwithin{Definition}{\DefaultNumberTheoremWithin}
     \numberlike{Definition}{Lemma}
     \labelformat{Definition}{Definition~#1}

  \theoremstyle{definition}
  \newtheorem{Question}{Question}
     \numberwithin{Question}{\DefaultNumberTheoremWithin}
     \numberlike{Question}{Lemma}
     \labelformat{Question}{Question~#1}

  \theoremstyle{definition}
  
     \numberwithin{Problem}{\DefaultNumberTheoremWithin}
     \numberlike{Problem}{Lemma}
     \labelformat{Problem}{Problem~#1}

  \theoremstyle{definition}
  \newtheorem{Remark}{Remark}
     \numberwithin{Remark}{\DefaultNumberTheoremWithin}
     \numberlike{Remark}{Lemma}
     \labelformat{Remark}{Remark~#1}

  \newtheorem{Example}{Example}
     \numberwithin{Example}{\DefaultNumberTheoremWithin}
     \numberlike{Example}{Lemma}
     \labelformat{Example}{Example~#1}
  
     \labelformat{Case}{Case~#1}
     \numberwithin{Case}{Lemma}
  
     \labelformat{Step}{Step~#1}
     \numberwithin{Step}{Lemma}

  \labelformat{equation}{(#1)}
  \labelformat{figure}{Figure~#1}
  \labelformat{section}{Section~#1}
  \labelformat{subsection}{Section~#1}
  \labelformat{subsubsection}{Section~#1}

  \def\eqref{\ref}

\begin{document}

\title[LSS-ideals and coordinate sections]{Lov\'asz-Saks-Schrijver ideals and   coordinate sections of  determinantal varieties}

\author[A.~Conca]{Aldo~Conca}
\address{Dipartimento di Matematica, 
Universit\`a di Genova, Via Dodecaneso 35, 
I-16146 Genova, Italy}
\email{conca@dima.unige.it}

\author[V.~Welker]{Volkmar Welker}
\address{Philipps-Universit\"at Marburg, 
         Fachbereich Mathematik und Informatik,
         D-35032 Marburg, Germany}
\email{welker@mathematik.uni-marburg.de}

\date{\today}
\thanks{The work on this paper was partly supported by a DAAD Vigoni project and by INdAM} 


\subjclass[2010]{Primary  }

%\date{\today}
 
%------------------------------------------------------------------------
%
%
%------------------------------------------------------------------------

\begin{abstract}
  Motivated by questions in algebra and combinatorics we study two ideals
  associated to a simple graph $G$: 
  \begin{itemize}
     \item the Lov\'asz-Saks-Schrijver ideal defining the 
           $d$-dimensional orthogonal representations  of the graph complementary 
           to $G$ and 
     \item the determinantal ideal of the $(d+1)$-minors of a generic symmetric matrix
             with $0$s 
           in positions prescribed by the  graph $G$. 
  \end{itemize}
  In characteristic $0$ these two ideals turn out to be closely related and algebraic 
  properties such as being radical, prime or 
  a complete intersection transfer from the 
  Lov\'asz-Saks-Schrijver ideal   to 
  the determinantal ideal. 
  For Lov\'asz-Saks-Schrijver ideals we link these properties to 
  combinatorial properties of $G$ and show that they always hold for $d$ 
  large enough. 
  For specific classes of graphs, such a forests, we can give a complete
  picture and classify the radical, prime and complete intersection 
  Lov\'asz-Saks-Schrijver ideals. 
\end{abstract}

\maketitle

\section{Introduction} 
\label{sec:intro}

Let $\KK$ be a field, $n \geq 1$ be an integer and 
set $[n] = \{1,\ldots, n\}$. For a simple graph $G=([n], E)$ with vertex set $[n]$ 
and edge set $E$ we study the following two classes of ideals associated 
to $G$. 

\begin{itemize}
  \item {\sf Lov\'asz-Saks-Schrijver ideals:}

    \noindent For an integer $d \geq 1$ we consider the polynomial ring 
    $S=\KK[y_{i,\ell}~|~i \in [n], \ell \in [d]]$. 
    For every edge $e=\{ i,j\} \in {[n] \choose 2}$ we set 
    $$f_e^{(d)}= \sum_{\ell =1}^d y_{i\ell}\,y_{j\ell}.$$ 
    The ideal 
		$$L_G^\KK(d) = (\,f_e^{(d)} ~|~ e\in E\,) \subseteq S$$ 
    is called the Lov\'asz-Saks-Schrijver ideal, LSS-ideal for short,
    of $G$ with respect to $\KK$.
    The ideal $L_G^\KK(d)$ defines the variety of orthogonal representations
		of the graph complementary to $G$. We refer the reader to
		\cite{LSS,L} for background on orthogonal representations and 
                results on the geometry of the variety of orthogonal
		representations which provided intuition for some of our 
		results.

  \item {\sf Coordinate sections of generic (symmetric) determinantal 
		ideals:}

    \noindent Consider the polynomial ring 
    $S=\KK[x_{ij}~|~ 1\leq i\leq j\leq n]$ and let 
    $X$ be the generic $n\times n$ symmetric matrix, that is, 
    the $(i,j)$-th entry of $X$ is $x_{ij}$ if $i\leq j$ and $x_{ji}$ if 
    $i > j$. 
    Let $X_G^\symmetric$ be the matrix obtained from $X$ by replacing the entries in 
		positions $(i,j)$ and $(j,i)$ for $\{i,j\} \in E$ with $0$. 
    For an integer $d$ let 
    $I_d^\KK(X_G^\symmetric) \subseteq S$ be the ideal 
    of $d$-minors of $X_G^\symmetric$.        
    The ideal $I_d^\KK(X_G^\symmetric)$ defines a coordinate hyperplane 
    section of the generic symmetric determinantal variety. Similarly, 
    we consider ideals defining coordinate hyperplane sections of 
    the generic determinantal varieties and the generic skew-symmetric 
    Pfaffian varieties. 
\end{itemize}

We observe in \ref{sec:minorgraph} that the ideal 
$I_{d+1}^\KK(X_G^\symmetric)$ and the ideal $L_G^\KK(d)$ are closely related. 
Indeed, if $\KK$ has characteristic $0$ classical results from 
invariant theory can be employed to show that $I_{d+1}^\KK(X_G^\symmetric)$ is 
radical (resp.~is prime, resp.~has the expected height) provided 
$L_G^\KK(d)$ is radical (resp.~is prime, resp.~is a complete intersection).
We also exhibit similar relations between variants of $L_G^\KK(d)$ 
and ideals defining coordinate sections of determinantal and Pfaffian ideals.

These facts turn the focus on algebraic properties of 
the LSS-ideals $L_G^\KK(d)$. 
In particular, we analyze the questions:  
when is $L_G^\KK(d)$ a radical ideal?  when is it a complete intersection?  
when is it a prime ideal? Other properties of ideals such as defining a
normal ring or a UFD are interesting as well but will not be treated here.  
In \ref{sec:stabilization} we prove the following: 

\begin{Theorem} 
  \label{thm:prime_and_ci} 
  Let $G = ([n],E)$ be a graph. Then:
  \begin{itemize} 
     \item[(1)] If $L_G^\KK(d)$ is prime then $L_G^\KK(d)$ is a 
       complete intersection.
     \item[(2)] If $L_G^\KK(d)$ is a complete intersection then 
       $L_G^\KK(d+1)$ is prime.
  \end{itemize} 
\end{Theorem} 

As an immediate consequence we have: 

\begin{Corollary} 
  \label{cor:prime_and_ci} 
  Let $G = ([n],E)$ be a graph. Then:
  \begin{itemize} 
      \item[(1)] If $L_G^\KK(d)$ is prime (resp.~complete intersection) then $L_G^\KK(d+1)$ is prime (resp.~complete intersection).
     \item[(2)] If $L_G^\KK(d)$ is prime (resp.~complete intersection)  then 
       $L_{G'}^\KK(d)$ is prime (resp.~complete intersection)   for every subgraph 
       $G'$ of $G$. 
      \end{itemize} 
\end{Corollary} 

In \ref{sec:pmd} we use these results to show that for $d$ large enough 
$L_G^\KK(d)$ is radical, prime and a complete intersection. To this end, 
for a graph $G=([n],E)$ we define a graph theoretic invariant 
$\pmd(G) \in \NN$, called the positive matching decomposition number of $G$. 
We prove in \ref{pmd} that $\pmd(G)\leq \min\{2n-3, |E|\}$ and that 
$\pmd(G)\leq \min\{n-1, |E|\}$ if $G$ is bipartite. We  show the 
following: 
 
\begin{Theorem} 
\label{thm:pmd} 
   Let $G = ([n],E)$ be a graph.   Then for  $d\geq \pmd(G)$ the ideal $L_G^\KK(d)$  
      is a radical complete intersection.  In particular, $L_G^\KK(d)$  is prime if  $d\geq \pmd(G)+1$. 
\end{Theorem} 

To have an explicit bound is crucial in order 
to use this result and the connection between $I_{d+1}^\KK(X_G^\symmetric)$ and $L_G^\KK(d)$. 
Indeed, for deducing meaningful results, we need to single out cases where we can say something about 
$L_G^\KK(d)$ for $d \leq n-1$. 
The results described in the following paragraph can be seen as steps in this 
direction.

Already in \ref{sec:stabilization} we give necessary conditions for 
$L_G^\KK(d)$ to be prime  in terms of subgraphs of $G$, see \ref{obsKab}.
In particular, we prove that if $L_G^\KK(d)$ is prime then $G$ does not contain a complete bipartite subgraph $K_{a,b}$ 
with $a+b=d+1$ (i.e. $\bar{G}$ is $(n-d)$-connected).  Similar results are obtained for complete intersections.  
In \ref{sec:6} we show that while these conditions in general are only necessary
for small values of $d$ they can be used to characterize the properties.
For $d=1$ the characterization is obvious and in \cite{HMMW} it is proved 
that $L_G^\KK(2)$  is prime if and only if $G$ is a matching. 
We obtain the following: 

\begin{Theorem}
  \label{thm:prim3}
  Let $G$ be a graph. Then:
  \begin{itemize}
    \item[(1)] $L_G^\KK(3)$ is prime if and only if $G$  does not contain   $K_{1,3}$ and does not contain  $K_{2,2}$. 
    \item[(2)] $L_G^\KK(2)$ is a complete intersection  if and only if $G$  does not contain   $K_{1,3}$  and does not 
            contain $C_{2m}$ for some $m\geq 2$.     
  \end{itemize}
\end{Theorem}

Here  $C_n$ denotes the cycle with $n$ vertices.  Finally for forests (i.e. graphs without cycles) we can give a  
complete picture.

\begin{Theorem}
  \label{thm:forest}
  Let $G$ be a forest  and denote by $\Delta(G)$  the maximal
degree of a vertex in $G$.  Then: 
  \begin{itemize}
    \item[(1)] $L_G^\KK(d)$ is radical for all $d$.
    \item[(2)] $L_G^\KK(d)$ is a complete intersection if and only if 
	       $d\geq \Delta(G)$.
    \item[(3)] $L_G^\KK(d)$ is prime if and only if $d\geq \Delta(G)+1$.  
  \end{itemize}  
\end{Theorem}

In \ref{sec:minorgraph} we demonstrate in characteristic $0$ the above 
mentioned connection between $L^\KK_G(d)$ and $I_{d+1}^\KK(X_G^\symmetric)$.
Using the results from the preceding sections we deduce sufficient conditions 
for $I_{d+1}^\KK(X_G^\symmetric)$ to be radical, prime or of expected height. 
Similar results are obtained for coordinate hyperplane sections of the generic 
determinantal varieties and the generic skew-symmetric Pfaffian varieties. 
To our knowledge coordinate sections of determinantal varieties have been 
systematically studied only in the case of maximal minors, for example the results in 
\cite{Bo,E,GM}.  

In \ref{sec:obstruction} we use the results from \ref{sec:stabilization} 
and \ref{sec:minorgraph} to formulate obstructions that prevent $L^\KK_G(d)$ 
to be prime or a complete intersection. We also study the
exact asymptotics in terms of the number of vertices of the least $d$ such 
that $L_G^\KK(d)$ is prime for $G$ a complete and a complete bipartite graph. 
Finally, in \ref{sec:questions} we pose open problems, formulate
conjectures and exhibit a relation between hypergraph LSS-ideals and coordinate 
sections of bounded rank tensor varieties. 

To complete the outline of the paper we mention that \ref{sec:generalities} 
sets up the graph theory and Gr\"obner theory notation. \ref{sec:results} 
recalls results from \cite{HMMW} for the case $d = 2$ which in particular show 
that $L_G^\KK(2)$ is always radical if $\chara \KK \neq 2$. 
We then exhibit and 
discuss counterexamples which demonstrate that this is not the case for $d=3$.

\medskip  

\noindent {\bf Acknowledgment}: we thank Alessio D'Ali, Alessio Sammartano and  Lorenzo Venturello for useful discussion concerning the material presented. 
%We thank also an anonymous referee for  several useful suggestions and comments concerning the earlier version of the paper. 

\section{Notations and generalities} 
\label{sec:generalities}

\subsection{Graph and Hypergraph Theory} 
\label{sec:graphtheory}
In the following we introduce graph theory notation. We mostly follow the
conventions from \cite{D}.
For us a graph $G = (V,E)$ is a simple graph on a finite vertex set $V$.
In particular, $E$ is a subset of the set of $2$-element subsets 
${V \choose 2}$ of $V$. In most of the cases we assume that 
$V=[n]= \{1,\ldots,n\}$. 
A subgraph of a graph $G = (V,E)$ is a graph $G' = (V',E')$ such that
$V' \subseteq V$ and $E' \subseteq E$.
Given two graphs $G$ and $G'$ we say that $G$ contains $G'$ if $G$ has a 
subgraph isomorphic to $G'$.  

More generally, a hypergraph $H = (V,E)$ is a pair consisting of a finite set
of vertices $V$ and a set $E$ of subsets of $V$. We are only interested 
in the situation when the sets in $E$ are inclusionwise incomparable.
Such a set of subsets is called a clutter.

For $m,n>0$ we will use the following notations:  

\begin{itemize}
  \item $K_n$ denotes  the complete graph on $n$ vertices, i.e. 
     $K_n=([n],\{ \{i,j\} : 1\leq i<j\leq n\})$, 
  \item $K_{m,n}$ denotes the complete bipartite graph 
     $([m] \cup [\tilde{n}], \{ \{i,\tilde{j}\} ~:~i \in [m], \tilde{j}
     \in [\tilde{n}]~\}$  with bipartition 
     $[m]$ and $[\tilde{n}] = \{ \tilde{1},\ldots, \tilde{n}\}$. 
  \item $B_n$ denotes the subgraph of $K_{n,n}$
     obtained by removing the edges $ \{i, \tilde{i} \}$ for $i=1,\dots, n$. 
  \item For $n>2$ we denote by $C_n$ the cycle with $n$ vertices, i.e. the subgraph of 
     $K_{n}$ with edges $\{1,2\},\{2,3\},\dots,$ $\{n-1,n\}, \{n,1\}$. 
  \item For $n>1$ we denote by $P_n$ the path with $n$ vertices, i.e. the 
     subgraph of $K_{n}$ with edges 
     $\{1,2\},\{2,3\},\dots, \{n-1,n\}$. 
\end{itemize}
 
We denote by $\bar{G} = (V,\bar{E})$ with 
$\bar{E} = {V \choose 2} \setminus E$ the graph complementary to $G = (V,E)$.
Let $W \subseteq V$. We write $G_W = (W,\{e\in E~:~e \subseteq W\})$
for the graph induced by $G$ on vertex set $W$ and $G-W$ 
for the subgraph induced by $G$ on $V \setminus W$. In case
$W = \{v\}$ for some $v \in V$ we simply write $G-v$ for $G-\{v\}$. 

A graph $G=([n],E)$  with $n \geq k+1$ is 
called $k$-(vertex)connected if 
for every $W\subset V$ with $|W|=k-1$ the graph $G-W$ is connected.  
The degree $\deg(v)$ of a vertex $v$ of $G$ is   $|\{ e \in E~|~v \in e\}|$  and 
  $\Delta(G)= \max_{v \in V} \deg(v)$.  
Clearly, if $G=([n],E) $ is $k$-connected then every vertex has degree
at least $k$ and $\Delta(\bar{G}) \leq n-k-1$. 
We denote by   $\omega(G)$  the clique number of $G$, i.e.  the largest $a$ such that
$G$ contains $K_a$.  The following well known fact follows directly from the definitions.

\begin{Lemma}
  \label{lem:n-d}
  Given a graph $G=([n],E)$ and an integer $1 \leq d \leq n$ the 
  following conditions are 
  equivalent: 
  \begin{itemize} 
    \item[(1)] $\bar{G}$ is $(n-d)$-connected. 
    \item[(2)] $G$ does not contain  $K_{a,b}$ with $a+b=d+1$. 
  \end{itemize}
\end{Lemma}

\subsection{Basics on LSS-ideals and their generalization to hypergraphs}
\label{sec:lssbasics}

Let $H = ([n],E)$ be a hypergraph.  For an integer $d \geq 1$ we consider the polynomial ring 
    $S=\KK[y_{i,\ell}~|~i \in [n], \ell \in [d]]$. We define for $e \in E$
$$f_e^{(d)}= \sum_{\ell =1}^d \prod_{i \in e} y_{i\ell}.$$ 
If $E$ is a clutter we call the ideal 
$$L_H^\KK(d) = (\, f_e^{(d)} ~|~ e\in E \,) \subseteq S$$ 
the LSS-ideal of the hypergraph $H$.

It will sometimes be useful to consider $L_H^\KK(d)$ as a 
multigraded ideal. For that we equip 
$S$ with the 
multigrading induced by $\deg(y_{i,\ell}) = \ee_i$ for the $i$-th unit vector
$\ee_i$ in $\ZZ^n$ and $(i,\ell) \in [n] \times [d]$. 
Clearly, for $e \in E$ the polynomial $f_e^{(d)}$ is multigraded
of degree $\sum_{i \in e} \ee_i$. In particular,  
$L_H^\KK(d)$ is $\ZZ^n$-multigraded. 
The following remark is an immediate consequence of the fact that if
$E$ is a clutter 
the two polynomials $f_e^{(d)}$ and $f_{e'}^{(d)}$ 
corresponding to distinct edges $e,e' \in E$ 
have incomparable multidegrees. 

\begin{Remark}
  \label{rem:minimal}
  Let $H = ([n],E)$ be a hypergraph such that $E$ is clutter.
  The generators $f_e^{(d)}$, $e \in E$, of $L_H^\KK(d)$ form a minimal system 
  of generators. In particular, $L_H^\KK(d)$ is a complete intersection if and 
  only if the polynomials $f_e^{(d)}$, $e \in E$, form a regular sequence. 
\end{Remark}

The following alternative description of $L_G^\KK(d)$ for a graph $G$ 
turns out to be helpful in some places.

\begin{Remark} 
  \label{rem:LSSasDet}
  Let $G=([n],E)$ be a graph. 
  Consider the $n\times d$ matrix $Y=( y_{i,\ell} )$. Then $L_G^\KK(d)$ is 
  the ideal generated by the entries of the matrix $YY^T$
  in positions $(i,j)$ with $\{i,j\}\in E$. 
  Here $Y^T$ denotes the transpose of $Y$. 

  Similarly, for a bipartite graph $G$, say a subgraph of $K_{m,n}$, one 
  considers two sets of variables $y_{ij}$ with $(i,j)\in [m]\times [d]$,
	$z_{ij}$ with $(i,j)\in [d]\times [n]$ and  
  the matrices 
  $Y=(y_{ij})$ and $Z=(z_{ij})$. 
  Then $L_G^\KK(d)$ coincides (after renaming the variables in the obvious way) with the ideal 
  generated by the entries  of the product matrix $YZ$
  in positions $(i,j)$ for $\{i,\tilde{j}\}\in E$. 
\end{Remark}

\subsection{Gr\"obner Bases} 
We use the following notations and facts from Gr\"obner bases theory, see for 
example \cite{BC}. 
Consider the polynomial ring $S=\KK[x_1,\dots,x_m]$.
For a vector $\ww = (w_i : i\in [m])\in \RR^m$
and a non-zero polynomial 
$$f = \displaystyle{\sum_{\alpha \in \NN^{[m]}} a_\alpha x^\alpha}$$ we set 
$m_\ww (f) = \max_{a_\alpha \neq 0} \{ \alpha \cdot \ww \}$
and 
$$\ini_\ww(f) = \sum_{\alpha \cdot \ww = m_\ww(f)} a_\alpha  x^\alpha.$$
The latter is called the initial form  of $f$ with respect to $\ww$. 
For an ideal $I$ we denote by  $\ini_\ww(I)$ the ideal generated by $\ini_\ww(f)$ with $f\in I\setminus\{0\}$. 
For a   term order $\prec$ we denote similarly by $\ini_\prec(f)$ the largest term of $f$ and by  $\ini_\prec(I)$ the ideal generated by $\ini_\prec(f)$ with $f\in I\setminus\{0\}$.   
The following will allows us to deduce properties of ideals from properties of
their initial ideals.

\begin{Proposition}
  \label{prop:transfer}  
  Let $I$ be a homogeneous ideal in the polynomial ring $S$ and let $\tau$ be either a term order  $\prec$  or a vector $\ww \in \RR^{m}$.
   If $\ini_\tau(I)$ is radical  (resp.~a
  complete intersection, resp.~prime) then so is $I$. 
  Moreover, if $I=(f_1,\dots, f_r)$  and the elements  $\ini_\tau(f_1),\ldots, \ini_\tau(f_r)$    form a regular sequence   then $f_1,\ldots, f_r$ form a regular sequence and  
  $\ini_\tau(I)=( \ini_\tau(f_1),\ldots, \ini_\tau(f_r) )$.   
  \end{Proposition}

\section{Known results and counterexamples for Lov\'asz-Saks-Schrijver ideals}
\label{sec:results}

We recall results  from \cite{HMMW} and present examples showing that 
$L_G^\KK(3)$ is not radical in general.  First observe that, for obvious 
reasons, $L_G^\KK(1)$ is radical, it is a complete intersection if and only 
if $G$ is a matching and it is prime if and only if $G$ has no edges. 
For $d =2$ the following result from \cite{HMMW} gives a complete answer for 
two of the three properties under discussion. 

\begin{Theorem}[Thm.1.1, Thm.1.2, Cor.5.3 in  \cite{HMMW} ]
  \label{HMMW}
  Let $G = ([n],E)$ be a graph.
  If $\chara \KK\neq 2$ then the ideal $L_G^\KK(2)$ is radical. 
  If $\chara \KK= 2$ then  $L_G^\KK(2)$ is radical if and only if $G$ is 
    bipartite. Furthermore, $L_G^\KK(2)$ is prime if and only if $G$ 
	is a matching. 
\end{Theorem} 

Indeed, in \cite{HMMW} the characterization of the graphs $G$ for which 
$L_G^\KK(2)$ is prime is given under the assumption that $\chara \KK\neq 1,2$ 
mod $(4)$ but it  turns out that the statement holds as well in arbitrary 
characteristic (see \ref{obsKab} for the missing details). 

The next examples show that $L_G^\KK(3)$ need not be 
radical. In the examples we assume that $\KK$ has characteristic $0$ but  
we consider it very likely that the ideals are not radical over any field.

A quick criterion implying that an ideal $J$ in a ring $S$ is not 
radical is to identify an element $g\in S$ such that  $J:g\neq J:g^2$.  
We call such a $g$ a witness (of the fact that $J$ is not radical). 
Of course the potential witnesses must be sought among the elements that are 
 ``closely related''  to $J$. Alternatively, one can try to 
compute the radical of $J$ or even its primary decomposition directly and 
read off whether $J$ is radical.  
But these direct computations are extremely time consuming for LSS-ideals 
and did not terminate on our computers in the examples below. 
Nevertheless, in all examples we have quickly identified witnesses. 
  
\begin{figure}

  ~\vskip-3cm
  \begin{tabular}{ccc}
    \begin{minipage}[t]{0.3\textwidth}
      \begin{picture}(0,0)%
	  \includegraphics[width=\textwidth]{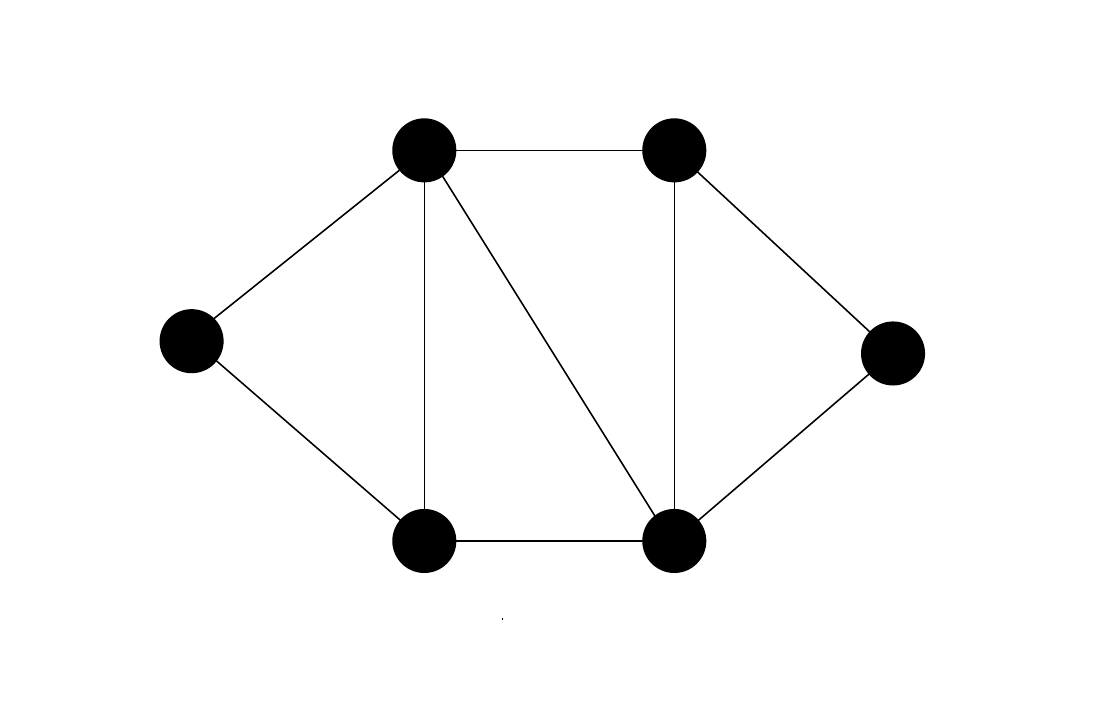}%
      \end{picture}%
      \setlength{\unitlength}{3947sp}%
      \begingroup\makeatletter\ifx\SetFigFont\undefined%
        \gdef\SetFigFont#1#2#3#4#5{%
        \reset@font\fontsize{#1}{#2pt}%
        \fontfamily{#3}\fontseries{#4}\fontshape{#5}%
        \selectfont}%
      \fi\endgroup%

      \begin{picture}(5349,3425)(1414,-3148)
        \put(1550,950){\makebox(0,0)[lb]{\smash{{\SetFigFont{12}{24.0}{\rmdefault}{\mddefault}{\updefault}{\color[rgb]{0,0,0}$5$}%
}}}}
        \put(2050,450){\makebox(0,0)[lb]{\smash{{\SetFigFont{12}{24.0}{\rmdefault}{\mddefault}{\updefault}{\color[rgb]{0,0,0}$3$}%
}}}}
        \put(2050,1450){\makebox(0,0)[lb]{\smash{{\SetFigFont{12}{24.0}{\rmdefault}{\mddefault}{\updefault}{\color[rgb]{0,0,0}$1$}%
}}}}
        \put(2550,1450){\makebox(0,0)[lb]{\smash{{\SetFigFont{12}{24.0}{\rmdefault}{\mddefault}{\updefault}{\color[rgb]{0,0,0}$4$}%
}}}}
        \put(2550,450){\makebox(0,0)[lb]{\smash{{\SetFigFont{12}{24.0}{\rmdefault}{\mddefault}{\updefault}{\color[rgb]{0,0,0}$2$}%
}}}}
        \put(3200,950){\makebox(0,0)[lb]{\smash{{\SetFigFont{12}{24.0}{\rmdefault}{\mddefault}{\updefault}{\color[rgb]{0,0,0}$6$}%
}}}}
      \end{picture}%

	    ~\vskip-7.7cm\centerline{(1)} 	     
    \end{minipage}
  & 
    \begin{minipage}[t]{0.3\textwidth}
      \begin{picture}(0,0)(0,-15)%
	\includegraphics[width=0.9\textwidth]{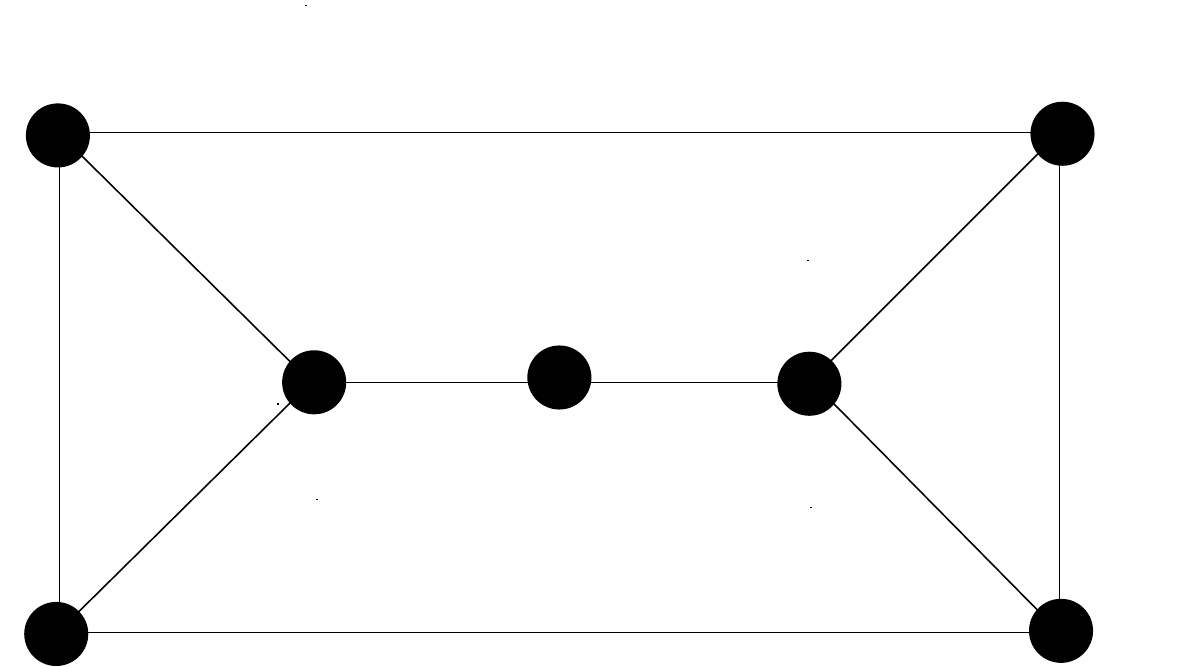}%
      \end{picture}%
      \setlength{\unitlength}{3947sp}%
      \begingroup\makeatletter\ifx\SetFigFont\undefined%
        \gdef\SetFigFont#1#2#3#4#5{%
        \reset@font\fontsize{#1}{#2pt}%
        \fontfamily{#3}\fontseries{#4}\fontshape{#5}%
        \selectfont}%
        \fi\endgroup%
      \begin{picture}(5349,3425)(1414,-3148)
        \put(1300,-2000){\makebox(0,0)[lb]{\smash{{\SetFigFont{12}{24.0}{\rmdefault}{\mddefault}{\updefault}{\color[rgb]{0,0,0}$1$}%
}}}}
       \put(3200,-2000){\makebox(0,0)[lb]{\smash{{\SetFigFont{12}{24.0}{\rmdefault}{\mddefault}{\updefault}{\color[rgb]{0,0,0}$2$}%
}}}}
        \put(3200,-3000){\makebox(0,0)[lb]{\smash{{\SetFigFont{12}{24.0}{\rmdefault}{\mddefault}{\updefault}{\color[rgb]{0,0,0}$3$}%
}}}}
        \put(1300,-3000){\makebox(0,0)[lb]{\smash{{\SetFigFont{12}{24.0}{\rmdefault}{\mddefault}{\updefault}{\color[rgb]{0,0,0}$4$}%
}}}}
        \put(1700,-2500){\makebox(0,0)[lb]{\smash{{\SetFigFont{12}{24.0}{\rmdefault}{\mddefault}{\updefault}{\color[rgb]{0,0,0}$5$}%
}}}}
        \put(2250,-2350){\makebox(0,0)[lb]{\smash{{\SetFigFont{12}{24.0}{\rmdefault}{\mddefault}{\updefault}{\color[rgb]{0,0,0}$6$}%
}}}}
        \put(2800,-2500){\makebox(0,0)[lb]{\smash{{\SetFigFont{12}{24.0}{\rmdefault}{\mddefault}{\updefault}{\color[rgb]{0,0,0}$7$}%
}}}}
      \end{picture}%

      \centerline{(2)}
    \end{minipage}
  & 
    \begin{minipage}[t]{0.3\textwidth}

	    \begin{picture}(0,0)(0,-15)%
	\includegraphics[width=0.9\textwidth]{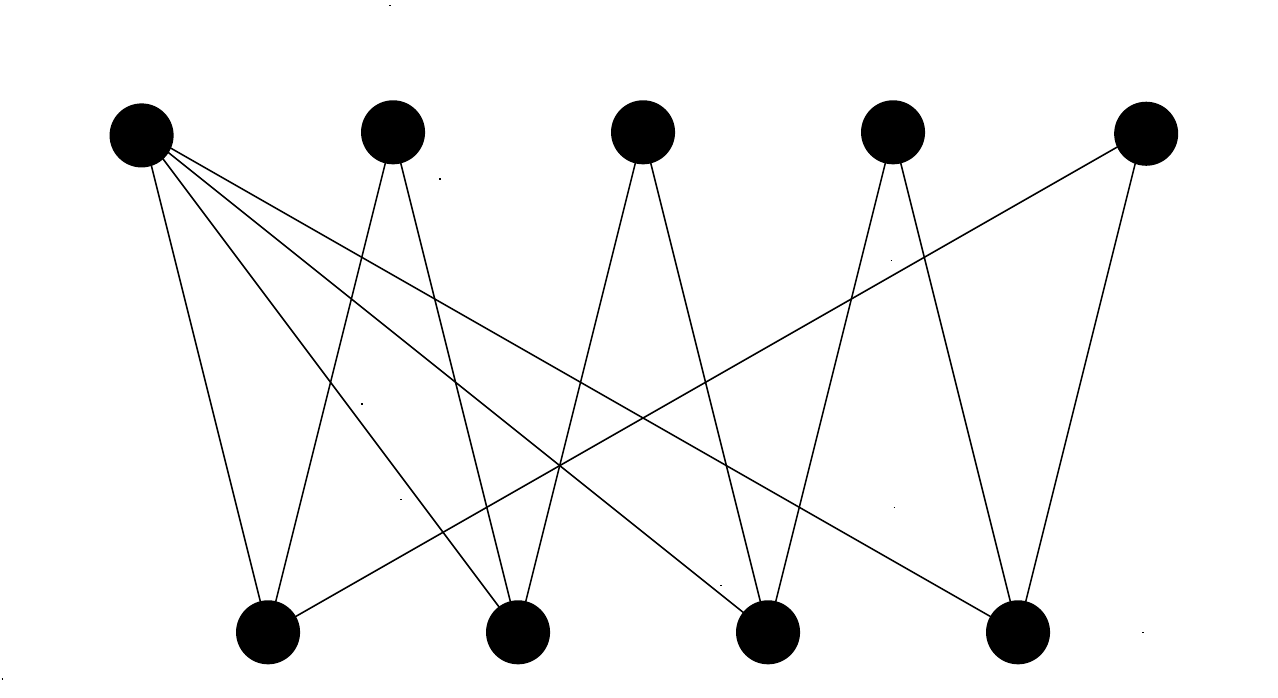}%
      \end{picture}%
      \setlength{\unitlength}{3947sp}%
      \begingroup\makeatletter\ifx\SetFigFont\undefined%
        \gdef\SetFigFont#1#2#3#4#5{%
        \reset@font\fontsize{#1}{#2pt}%
        \fontfamily{#3}\fontseries{#4}\fontshape{#5}%
        \selectfont}%
        \fi\endgroup%

      \begin{picture}(6129,3335)(2914,-3676)
        \put(2975,840){\makebox(0,0)[lb]{\smash{{\SetFigFont{12}{24.0}{\rmdefault}{\mddefault}{\updefault}{\color[rgb]{0,0,0}$1$}%
}}}}
        \put(3350,840){\makebox(0,0)[lb]{\smash{{\SetFigFont{12}{24.0}{\rmdefault}{\mddefault}{\updefault}{\color[rgb]{0,0,0}$2$}%
}}}}
        \put(3725,840){\makebox(0,0)[lb]{\smash{{\SetFigFont{12}{24.0}{\rmdefault}{\mddefault}{\updefault}{\color[rgb]{0,0,0}$3$}%
}}}}
        \put(4100,840){\makebox(0,0)[lb]{\smash{{\SetFigFont{12}{24.0}{\rmdefault}{\mddefault}{\updefault}{\color[rgb]{0,0,0}$4$}%
}}}}
        \put(4475,840){\makebox(0,0)[lb]{\smash{{\SetFigFont{12}{24.0}{\rmdefault}{\mddefault}{\updefault}{\color[rgb]{0,0,0}$5$}%
}}}}
	      \put(3150,-200){\makebox(0,0)[lb]{\smash{{\SetFigFont{12}{24.0}{\rmdefault}{\mddefault}{\updefault}{\color[rgb]{0,0,0}$\tilde{1}$}%
}}}}
	      \put(3525,-200){\makebox(0,0)[lb]{\smash{{\SetFigFont{12}{24.0}{\rmdefault}{\mddefault}{\updefault}{\color[rgb]{0,0,0}$\tilde{2}$}%
}}}}
	      \put(3900,-200){\makebox(0,0)[lb]{\smash{{\SetFigFont{12}{24.0}{\rmdefault}{\mddefault}{\updefault}{\color[rgb]{0,0,0}$\tilde{3}$}%
}}}}
	      \put(4275,-200){\makebox(0,0)[lb]{\smash{{\SetFigFont{12}{24.0}{\rmdefault}{\mddefault}{\updefault}{\color[rgb]{0,0,0}$\tilde{4}$}%
}}}}
      \end{picture}%

      ~\vskip-7.7cm\centerline{(3)} 	     
    \end{minipage}
  \end{tabular}
  \medskip
  \caption{Graphs $G$ with non-radical $L_G^\KK(3)$}
  \label{fig:nrad1}
\end{figure}
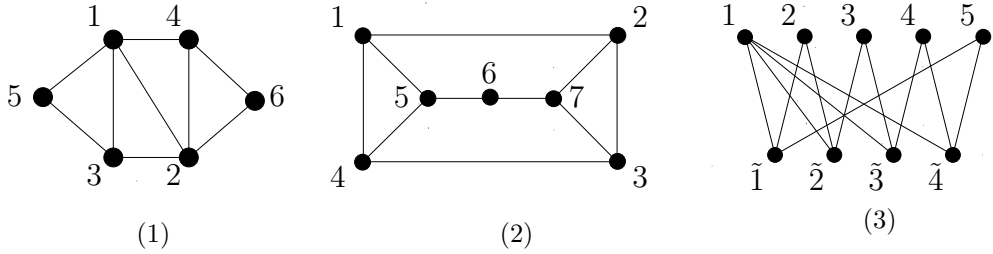

\begin{Example}
  \label{ex:nrad}
  We present three examples of graphs $G$ such that $L_G^\KK(3)$ is not 
  radical over any field $\KK$ of characteristic $0$. 
  The first example has $6$ vertices and $9$ edges and it is the smallest 
  example we have found (both in terms of edges and vertices). The second 
  example has $7$ vertices and $10$ edges and it is a complete intersection. 
  This shows that $L_G^\KK(3)$ can be a complete intersection without being 
  radical. The third example is bipartite, a subgraph of $K_{5,4}$, with $12$ 
  edges, and is the smallest bipartite example we have found. 
  In all cases, since the LSS-ideal $L_G^\KK(3)$ has integral coefficients,  
  we may assume that $\KK=\QQ$ and exhibit a witness $g$, i.e. a polynomial $g$ 
  such that $L_G^\KK(3):g\neq L_G^\KK(3):g^2$. The latter inequality can be 
  checked  with the help of CoCoA \cite{AB} or Macaulay 2 \cite{GS}.       

  \begin{itemize}
    \item[(1)] 
      Let $G$ be the  graph with $6$ vertices and $9$ edges depicted 
      in \ref{fig:nrad1}(1),  i.e. with edges 
      $$E=\{ \{1, 2\}, \{1, 3\}, \{1, 4\}, \{1, 5\}, \{2, 3\}, \{2, 4\}, 
             \{2, 6\}, \{3, 5\}, \{4, 6\}   \}.$$
 
      Here the witness can be chosen as follows. Denote by $Y=(y_{ij})$ 
      a generic $6\times 3$ matrix. As discussed in \ref{rem:LSSasDet} 
      the ideal $L^\QQ_G(3)$ is generated by the entries of $YY^{T}$ 
      corresponding to the positions in $E$. Now $g$ can be taken as the 
      $3$-minor of $Y$  with row indices $1,5,6$.
 
    \item[(2)] Let $G$ be the  graph with $7$ vertices and $10$ edges depicted
      in \ref{fig:nrad1}(2), i.e. with edges 
      $$E=\{ \{1, 2\}, \{1, 4\}, \{1, 5\}, \{2, 3\}, \{2, 7\}, \{3, 4\}, 
             \{3, 7\}, \{4, 5\}, \{5, 6\}, \{6, 7\}   \}.$$
      Here the witness can be chosen as follows. Denote by $Y=(y_{ij})$ 
      a generic $7\times 3$ matrix. Again as discussed in \ref{rem:LSSasDet} 
      the ideal $L^\QQ_G(3)$ is generated by the entries of $YY^{T}$ 
      corresponding to the positions in $E$. Now $g$ can be taken as the 
      $3$-minor of $Y$ with row indices $1,2,4$. The fact that $L^\QQ_G(3)$ 
      is a complete intersection can be checked quickly 
      with CoCoA \cite{AB} or Macaulay 2 \cite{GS}. 

    \item[(3)] Let $G$ be the subgraph of the complete bipartite graph 
      $K_{5,4}$ depicted in \ref{fig:nrad1}(3), i.e. with edges 
      $$E=\{ \{1,\tilde{1}\}, \{1,\tilde{2}\}, \{1,\tilde{3}\}, 
             \{1,\tilde{4}\}, \{2,\tilde{1}\}, \{2,\tilde{2}\}, 
             \{3,\tilde{2}\}, \{3,\tilde{3}\}, \{4,\tilde{3}\}, 
             \{4,\tilde{4}\}, \{5,\tilde{1}\}, \{5,\tilde{4}\} \}.
      $$
      Denote by $X=(x_{ij})$ a generic $5\times 3$ matrix and by $Y=(y_{ij})$ 
      a generic $3\times 4$ matrix. As explained in \ref{rem:LSSasDet} the 
      ideal $L^\QQ_G(3)$ is generated by the entries of $XY$ corresponding 
      to the positions in $E$.  Now the witness $g$ can be taken to be the 
      $3$-minor of $X$ corresponding to the column 
      indices $1,2,4$.
  \end{itemize}
\end{Example} 

\section{Stabilization of algebraic properties of $L_G^\KK(d)$}
\label{sec:stabilization}

In this section we prove \ref{thm:prime_and_ci} and state some of its 
consequences. We recall first   some facts on the symmetric algebra of a module stating 
 the results in the way that suit our needs best. 

Recall that, given a ring $R$ and an $R$-module $M$ presented as the cokernel 
of an $R$-linear map $$f:R^m  \to  R^n$$ 
the symmetric algebra $\Sym_R(M)$ 
of $M$ is (isomorphic to) the quotient of $\Sym_R(R^n)=R[x_1,\dots, x_n]$ 
by the ideal $J$ generated by the entries of $A\,(x_1,\dots,x_n)^T$ 
where $A$ is the $m\times n$ matrix representing $f$. Vice
versa every quotient of $R[x_1,\dots, x_n]$ by an ideal $J$ generated by 
homogeneous elements of degree $1$ in the $x_i$'s is the symmetric 
algebra of an $R$-module. 

Part (1) of the following is a special case of \cite[Prop. 3]{A} and 
part (2) a special case of  \cite[Thm 1.1]{Hu}.
Here and in the rest of the paper for a matrix $A$ with
entries in a ring $R$ and a number $t$  we denote by $I_t(A)$ the ideal of $R$ generated by
the $t$-minors of $A$. 

\begin{Theorem}
  \label{AHu}
  Let $R$ be a complete intersection. 
  Then 
  \begin{itemize}
    \item[(1)] $\Sym_R(M)$ is a complete intersection if and only if 
      $\height\,I_t(A) \geq m-t+1$ for all $t=1,\dots, m$. 
    \item[(2)] $\Sym_R(M)$ is a domain and $I_m(A)\neq 0$ if and only if 
      $R$ is a domain,  and $\height\,I_t(A)\geq m-t+2$ for all $t=1,\dots, m$. 
  \end{itemize}
  The equivalent conditions of (2) imply those of (1). 
\end{Theorem} 

\begin{Remark} 
  \label{Arema1}
  Let $G=([n],E)$ be a graph. The ideal $L^\KK_G(d)
  \subseteq S= \KK[y_{i,j} :  i\in [n], \ \ j\in [d] ]$ is generated 
  by elements that have degree at most one in each block of variables.
  Hence $L^\KK_G(d)$ can be seen as an ideal defining a symmetric algebra 
  in various ways. 

  For example, set $G_1=G-n$, $U=\{ i\in [n-1] | \{i,n\}\in E\}$, $u=|U|$,  
  $S'=\KK[y_{i,j} :  i\in [n-1], \ \ j\in [d] ]$ and $R=S'/L_{G_1}^\KK(d)$. 
  Then $S/L_G^\KK(d)$ is the symmetric algebra of the cokernel of the 
  $R$-linear map $$R^u\to R^d$$ associated to the $u\times d$ matrix 
  $A=(y_{ij})$ with $i\in U$ and $j=1,\dots,d$.  
\end{Remark} 

\begin{Remark} 
  \label{Arema2}
  In order to apply \ref{AHu} to the case described in \ref{Arema1} it is 
  important to observe that for every $G$ no minors of the matrix 
  $(y_{ij})_{(i,j) \in [n] \times [d]}$ vanish modulo $L_G^\KK(d)$. 
  This is because $L_G^\KK(d)$ is contained in the ideal $J$ generated by 
  the monomials $y_{ik}y_{jk}$ and the terms in the minors of $(y_{ij})$ do 
  not belong to $J$ for obvious reasons. 
\end{Remark}

\begin{Proposition}
\label{obsKab} 
 Let $G=([n],E)$ be a graph. If $L_G^\KK(d)$ is prime then  $G$ does not contain $K_{a,b}$ with  $a+b > d$.
 \end{Proposition} 
 \begin{proof}  Suppose by  contradiction that $L_G^\KK(d)$ is prime and 
  $G$ contains $K_{a,b}$ for some 
  $a+b> d$. We may decrease either $a$ or $b$ or both and 
  assume right away that $a+b=d+1$ with $a,b \geq 1$. 
  In particular $a,b \leq d$ and $a+b\leq n$.
  We may assume that  $K_{a,b}$ is a subgraph
  of $G$ with edges $\{i,a+j\}$ for $i \in [a]$ and
  $j \in [b]$. 
  Set $R = S/L_{G}^\KK(d)$ and $Y = (y_{i\,\ell}) \in R^{a\times d}$ and 
  $Z = (z_{\ell,i}) \in R^{d\times b}$ with $z_{\ell,i}=y_{i+a,\ell}$. 
  Since $K_{a,b}$ is a subgraph of $G$ we 
  have $YZ=0$ in $R$. 
By assumption $R$ is a domain and $YZ = 0$ can be seen as a 
  matrix identity over the field   of fractions of $R$.
  Hence 
  $$\rank(Y)+\rank(Z) \leq d.$$
  From $a+b = d+1$ it follows that $\rank(Y)<a$ or $\rank(Z)<b$.  
  This implies that   $I_a(Y)=0$ or $I_b(Z)=0$ as ideals of $R$.
  But by \ref{Arema2} none of the minors of $Y$ and $Z$ 
  are in  $L_G^\KK(d)$. This is a contradiction
  and hence $L_G^\KK(d)$ is not prime. 
  \end{proof}

\begin{Lemma} 
 \label{detcol} 
  Let $A$ be an $m\times n$ matrix with entries in a Noetherian ring $R$. 
  Assume $m\leq n$. Let $S=R[x]=R[x_1,\dots, x_m]$ be a polynomial ring over 
  $R$ and let $B$ be the $m\times (n+1)$ matrix with entries in $S$ 
  obtained by adding the column $(x_1,\dots,x_m)^T$ to $A$. 
  Then we have $\height\,I_1(B)=\height\,I_1(A)+m$ and 
  $$\height\,I_t(B)\geq \min\{ \height\, I_{t-1}(A) ,  
    \height\,I_{t}(A)+m-t+1\}$$ 
  for all $1<t\leq m$.  
\end{Lemma} 

\begin{proof} 
  Set $u= \min\{ \height\,I_{t-1}(A) , \height\,I_{t}(A)+m-t+1\}$. 
  Let $P$ be a prime ideal of 
  $S$ containing $I_t(B)$.  We have to prove that $\height\,P\geq u$. 
  If $P\supseteq I_{t-1}(A)$ then $\height\, P\geq \height\, I_{t-1}(A)\geq u$. 
  If $P\not\supseteq I_{t-1}(A)$ then we may assume that the $(t-1)$-minor
  $F$ corresponding to the first $(t-1)$ rows and columns of $A$ is not in $P$.
  Hence, $\height\, P=\height\, PR_F[x]$ and $PR_F[x]$ contains 
  $I_t(A)R_F[x]$ and $( x_j-F^{-1}G_j : j=t,\dots,m)$ with 
  $G_j\in R[x_1,\dots, x_{t-1}]$.  Since the elements 
  $x_j-F^{-1}G_j$ are algebraically 
  independent over $R_F$ we have  
  $$\height\, PR_F[x]\geq \height\, I_t(A)R_F+(m-t+1)\geq 
     \height\,I_t(A)+(m-t+1).$$
\end{proof}

Now we are in position to prove \ref{thm:prime_and_ci}: 

\begin{proof}[Proof of \ref{thm:prime_and_ci}] 
  To prove (1) we argue by induction on $n$.  The induction base $n\leq 2$ is obvious. 
  Assume $n > 2$. We use the notation from \ref{Arema1} and set 
  $S=\KK[y_{ij} : i\in [n], \ j\in [d] ]$,
  $S'=\KK[y_{i,j} :  i\in [n-1], \ \ j\in [d] ]$,
  $G_1 = G - n$,
  $U=\{ i\in [n-1] | \{i,n\}\in E\}$ and $u=|U|$.  
  Note, that $S'/L_{G_1}^\KK(d)$ is an algebra retract of $S/L_{G}^\KK(d)$. 
  Therefore $L_{G_1}^\KK(d)=
  L_{G}^\KK(d)\cap S'$ and so $L_{G_1}^\KK(d)$ is prime. 
  By induction, it follows that $L_{G_1}^\KK(d)$ is a complete intersection. 
  Since $u$ is the degree of the vertex $n$ in $G$  we have that  
  $K_{1,u}\subset G$. Since $L_{G}^\KK(d)$ is prime  \ref{obsKab} implies $1+u<d+1$, 
  i.e. $u<d$. By virtue of \ref{Arema2} we have that the minors of the 
  matrix $A$ are non-zero in $S'/L_{G_1}^\KK(d)$. In particular, 
  $I_u(A)\neq 0$ in $S'/L_{G_1}^\KK(d)$ and hence (2) in \ref{AHu} hold. 
  Then (1) in \ref{AHu} holds as well, i.e. $L_{G}^\KK(d)$ is a complete 
  intersection.

  To prove (2) we again argue by induction on $n$. For $n \leq 2$ the assertion is obvious. 
  Assume $n>2$. 
  We again use the notation $G_1=G-n$, $U=\{ i \in [n] | \{i,n\}\in E \}$,  $u=|U|$.
  In addition we set	
  $Y=( y_{ij})_{(i,j) \in U \times [d+1]}$, 
  $S=\KK[y_{ij} : i\in [n], \ j\in [d+1] ]$,
  $S'=\KK[y_{ij} : i\in [n-1], \ j\in [d+1] ]$ and $R=S'/L_{G_1}^\KK(d+1)$. 
  By construction, $S/L_G^\KK(d+1)$ is the symmetric algebra of the $R$-module 
  presented as the cokernel of the map $R^u\to R^{d+1}$ associated to $Y$. 
  
  By assumption, $L_G^\KK(d)$ is a complete intersection and hence 
  $L_{G_1}^\KK(d)$ is a complete intersection as well. It then follows by 
  induction that $L_{G_1}^\KK(d+1)$ is prime and hence $R$ is a domain.   
  Since the  polynomials  $f_{\{i,n\}}^{(d)}$ with $i\in U$ are a regular sequence contained 
  in the ideal $( y_{nj} : 1\leq j\leq d)$   we have $u\leq d$ and by \ref{Arema2} $I_u(Y)\neq 0$ 
  in $R$. Therefore, by \ref{AHu}(2) we have 
  $$L_{G}^\KK(d+1) \mbox{ is prime } \Leftrightarrow \height\, I_t(Y) 
    \geq u-t+2 \mbox{ in } R  \mbox{ for every } t=1,\dots, u.$$  

  Equivalently, we have to prove that 

  $$\height \Big( I_t(Y)+L_{G_1}^\KK(d+1) \Big) \geq u-t+2+g  \mbox{ in } S' 
    \mbox{ for every } t=1,\dots, u$$    
  where $g=\height\,L_{G_1}^\KK(d+1)=|E|-u$. 
 
  Consider the weight vector $\ww \in \RR^{n \times (d+1)}$ 
  defined by $\ww_{ij}=1$ and $\ww_{i\,d+1}=0$ 
  for all $j \in [d]$ and $i \in [n]$. 
  By construction the initial forms of the standard
  generators of $\ini_\ww(L_{G_1}^\KK(d+1))$ are the 
  standard generators of $L_{G_1}^\KK(d)$.  Since the standard generators of $I_t(Y)$ coincide with their initial  forms with respect to $\ww$ it follows that 
  $\ini_\ww(I_t(Y)) \supseteq I_t(Y)$ (indeed equality holds but we do not 
  need this fact). 

  Therefore, 
  $\ini_\ww( I_t(Y)+L_{G_1}^\KK(d+1))\supseteq I_t(Y)+L_{G_1}^\KK(d)$ 
  and it is enough to prove that

  $$\height \Big(I_t(Y)+L_{G_1}^\KK(d)\Big) \geq u-t+2+g  
    \mbox{ in } S' \mbox{ for every } t=1,\dots, u$$

  or, equivalently,   

  $$\height\, I_t(Y) \geq u-t+2 \mbox{ in } R' \mbox{ for every } t=1,\dots, u$$
  where $R'=S'/L_{G_1}^\KK(d)$. 

  The variables $y_{1\,d+1}, \dots, y_{n-1\,d+1}$ do not appear in the 
  generators of $L_{G_1}^\KK(d)$. 
  Hence $R'=R''[y_{1\,d+1}, \dots, 
  y_{n-1\,d+1}]$ with 
  $R''=\KK[y_{ij} : (i,j) \in [n-1] \times [d]]/L_{G_1}^\KK(d)$. 
  Let $Y'$ be the matrix $Y$ with the $(d+1)$-st column removed. 
  Then $S/L_G^\KK(d)$ can be regarded 
  as the symmetric  algebra of the $R''$-module presented as the cokernel of the map 
  \begin{eqnarray}
    \label{eq:co2}
       (R'')^u & \xrightarrow{Y'} & (R'')^{d}.
  \end{eqnarray}
  
  By assumption $S/L_G^\KK(d)$ is a 
  complete intersection.
  Hence by \ref{AHu}(1)  we know  
  $$\height\,I_t(Y') \geq u-t+1 \mbox{ in } R'' \mbox{ for every } 
    t=1,\dots, u$$

  Since $Y$ is obtained from $Y'$ by adding a column of variables over 
  $R''$ by \ref{detcol} we have: 

  $$\height\, I_t(Y) \geq \min\{  \height\,I_{t-1}(Y') , 
    \height \,   I_{t}(Y')+u-t+1 \} \geq u-t+2$$ 
  in $R'$ and for all $t=1,\dots,u$. 
\end{proof} 

Now we prove \ref{cor:prime_and_ci}: 

\begin{proof}[Proof of \ref{cor:prime_and_ci}]
Assertion (1) in \ref{cor:prime_and_ci} is a formal consequence of   \ref{thm:prime_and_ci}.  
Assertion (2) is obvious for complete intersections.  
Finally assume that $L^\KK_G(d)$ is prime. Then by \ref{thm:prime_and_ci} $L^\KK_G(d)$ is a complete intersection.   
The statement now follows from a general fact:  if a regular sequence generates a prime ideal in a standard 
graded algebra or in a local ring  then so does every subset of the sequence.  
\end{proof}

\section{Positive matching decompositions}
\label{sec:pmd}

In this section we introduce  positive matching decompositions and prove \ref{thm:pmd}.
  
\begin{Definition} 
  \label{def:posmat}
  Given a hypergraph $H=(V,E)$  
  a positive matching of $H$ is a subset 
  $M\subset E$ of pairwise disjoint sets (i.e., a matching) such that there 
  exists a weight function $w:V\to \RR$ satisfying:
  \begin{eqnarray}
    \label{eq:posmat}
    \begin{aligned} 
      \sum_{i \in A} w(i)>0 & \mbox{~if~} & A \in M \\
      \sum_{i \in A} w(i)<0 & \mbox{~if~} & A \in E\setminus M. 
    \end{aligned} 
  \end{eqnarray}
\end{Definition}
 
The next lemma summarizes some elementary properties of positive matchings.

\begin{Lemma}
  \label{lem:posmat}
  Let $H = (V,E)$ be a hypergraph such that $E$ is a clutter, 
  $M \subseteq E$ and $V_M = \bigcup_{A \in M} A$.
  \begin{itemize}
    \item[(1)] $M$ is a positive matching for $H$ if and
      only if $M$ is a positive matching for the induced hypergraph
      $(V_M, \{ A\in E~|~ A \subseteq V_M\})$.
    \item[(2)] 
      Assume $M$ is a positive matching on $H$ and 
      $A \in E$ is such that $M_1=M \cup \{A\}$ is a matching. Assume also there is a vertex $a\in A$ such that 
      $$\{ B \in E : \  B\subset V_{M_1} \mbox{ and }   a \in B \}=\{A\}.$$
      Then $M\cup\{A\}$   is a positive matching of $H$. 

    \item[(3)] If $H$ is a bipartite graph with bipartition 
      $V = V_1 \cup V_2$ then 
      $M$ is a positive matching if and only if $M$ is a matching and
      directing the edges $e \in E$ from $V_1$ to $V_2$ if $e \in M$ and
      from $V_2$ to $V_1$ if $e \in E \setminus M$ yields an acyclic 
      orientation.
  \end{itemize}
\end{Lemma}
\begin{proof}
  \begin{itemize}
     \item[(1)] Set $H_1=(V_M, \{ A\in E~|~ A \subseteq V_M\})$. Clearly a weight function on $V$ for which $M$ is a positive
	     matching restricts to $V_M$ making $M$ a positive matching of $H_1$.  
          Conversely, assume we are given a weight function $w$ on $V_M$ that makes $M$ a positive matching. Then we extends $w$ to $V$ by 
		  assigning to the vertices in $V \setminus V_M$ a weight sufficiently negative to induce a negative weight on the 
		  elements of $E$ which contain at least one element from $V \setminus V_M$. For example, one can set 
          $w(i)=-|V| \max \{w(j)~:~j \in V_M\}$ for every  $i \in V \setminus V_M$. Such an extension makes $M$ a positive matching for  $H$.
           
 \item[(2)] 
  Let $w$ be a weight that makes $M$ a positive matching of $H$. 
  In view of (1), it is enough to prove that there is a weight $v$ defined on $V_{M_1}$ making $M_1$ a positive matching for the 
  restriction of $H$ to $V_{M_1}$.  We set $v(i)=w(i)$ if $i\in V_{M_1}$ and $i\neq a$ and we give 
		  $v(a)$ a high enough value to have $v(A) > 0$, i.e. $v(a)> -\sum_{i\in A \ i\neq a} w(i)$.   Since there are no  
  elements in $E$ other than $A$ that are contained in $V_{M_1}$ and contain $a$ the resulting weight $v$ has the desired properties.
 
     \item[(3)] 
       We change the coordinates $w(i)$ to $-w(i)$ for $i \in V_2$ 
       in the inequalities defining a positive matchings. 
       As a simple reformulation of \ref{eq:posmat} we get that 
       in these coordinates a matching $M$ is 
       positive if and only if there is a weight function such that
       for $\{i,j\} \in E$, $i \in V_1$, $j \in V_2$ we have 
       \begin{eqnarray}
          \label{eq:bip}
          \begin{aligned}
          w(i) > w(j) & \mbox{~if~} & \{i,j\} \in M, \\
          w(i) < w(j) & \mbox{~if~} & \{i.j\} \in E \setminus M.
          \end{aligned}
       \end{eqnarray}
       This is equivalent to the existence of a region in the 
       arrangement of hyperplanes $w(i) = w(j)$ for $\{i,j \} \in E$
       in $\RR^V$ satisfying \eqref{eq:bip}. 
       But it is well known that the regions in this arrangement are in one
       to one correspondence with the acyclic orientations of $G$
       (see \cite[Lemma 7.1]{GZ}).
  \end{itemize}
\end{proof}

Now we are in position to introduce the key concept of this section.

\begin{Definition} 
  Let $H=(V,E)$ be a hypergraph for which $E$ is a clutter. 
  A positive matching decomposition 
  (or $\posMa$-decomposition) of $G$ is a partition 
  $E=\bigcup_{i=1}^p E_i$ into pairwise disjoint subsets such that 
  $E_i$ is a positive matching 
  on $(V, E\setminus \cup_{j=1}^{i-1} E_j)$ for $i=1,\dots, p$. 
  The $E_i$ are called the parts of the $\posMa$-decomposition.
  The smallest $p$ for which $G$ admits a $\posMa$-decomposition  
  with $p$ parts will be denoted by $\pmd(H)$.
\end{Definition}
  
Note that   one has $\pmd(H)\leq |E|$ because of the obvious 
$\posMa$-decomposition $\bigcup_{A\in E} \{A\}$. On the other hand $\pmd(G)$ 
is smaller than $|E|$ for most clutters. For graphs we have: 

\begin{Lemma}
\label{pmd}
Let $G=([n],E)$ be a graph. Then: 
\begin{itemize}
\item[(1)]  $\pmd(G)\leq \min(2n-3, |E|)$. 
\item[(2)] If $G$ is bipartite then $\pmd(G)\leq \min(n-1, |E|)$.  
\item[(3)]  $\pmd(G)\geq \Delta(G)$ with equality if $G$ is a forest.  
\end{itemize}  
\end{Lemma} 

\begin{proof}

	\begin{itemize}
  \item[(1)] 
    Since we have already argued that $\pmd(G)\leq |E|$ to prove the first statement we have show 
    $\pmd(G)\leq 2n-3$. To this end we may assume that $G$ is the complete graph $K_n$ because 
    any $\posMa$-decomposition of $K_n$ induces a $\posMa$-decomposition 
    on its subgraphs. For $\ell=1,\dots, 2n-3$ we set 
    $E_\ell=\{ \{i,j\} : i+j=\ell+2\}$.
    Clearly one has $E=\cup_{\ell=1}^{2n-3} E_\ell$. So to prove that this is a  
    $\posMa$-decomposition of $K_n$ we have to prove that $E_t$ is a positive matching on 
    $G_t=([n], \cup_{\ell=t}^{2n-3} E_\ell)$.  To this end  we build $E_t$ by inserting the edges 
    one by one starting from those that involve  vertices  with smaller indices  and  repeatedly use  
    \ref{lem:posmat} (2)  to prove that we actually get a positive matching. 
    For example for $n = 8$, to prove that $E_7$ is a positive matching on $G_7$ we order the elements in
    $E_7$ as follows $ \{4,5\},  \{3,6\},  \{2,7\}, \{1,9\}$. We assume we know already that 
    $ \{\{4,5\},  \{3,6\} \}$ is a positive matching and use  \ref{lem:posmat} (2) with $A=\{2,7\}$ and $a=2$ 
    to prove that $ \{\{4,5\},  \{3,6\}, \{2,7\}  \}$ is a positive matching matching as well.    

  \item[(2)] In this case it is enough to prove that 
  $\pmd(K_{m,n})\leq n+m-1$.  For $\ell =1,\ldots, m+n-1$ we $E_\ell = \{ \{i,\tilde{j}\}  :~i+j=\ell+1\}$. 
    Clearly one has $E = \bigcup_{\ell=1}^{m+n-1}  E_\ell$. So to prove that this is a  positive matching decomposition 
    of $K_{m,n}$ we have to prove that $E_\ell$ is a positive matching on $E \setminus 
    \bigcup_{k=1}^{\ell-1} E_k$ for $\ell = 1 ,\ldots, m+n-1$. 
        
    For $\ell = 1$  the assertion is obvious since  $E_1$ 
    contains a single edge.  
    Now assume $\ell \geq 2$. By \ref{lem:posmat}(3) it suffices to show 
    that directing the edges in $E_\ell$ from $[m]$ to $[\tilde{n}]$
    and the edges in $E \setminus \bigcup_{k=1}^\ell E_k$ in the other
    direction yields an acyclic orientation.
    Assume the resulting directed graph has a directed cycle. Let  $\{i,\tilde{j}\} \in E_\ell$ be the edge 
    from $E_\ell$ in this directed cycle for which $j$ is minimal. 
    The directed edge following the edge $i\rightarrow \tilde{j}$ in the directed cycle is 
    of the form $\tilde{j} \rightarrow i'$ for some $i'$ with $i'+j > \ell+1$. This implies $i' > i$. 
    Now let $i' \rightarrow \tilde{j'}$ be the edge following $\tilde{j} \rightarrow i'$ in the directed
    cycle. Then $\{ i',\tilde{j'}\} \in E_\ell$ and $i'+j' = \ell+1$. But this yields $j' < j$ which
    contradicts the minimality of $j$.
    Hence there is no directed cycle and $E_\ell$ is a positive matching on
    $E \setminus \bigcup_{k=1}^{\ell-1} E_k$.  
   
  \item[(3)] The inequality $\Delta(G)\leq \pmd(G)$ is obvious. To prove that equality holds if $G$ is a forest we argue by induction on the number of vertices. 
    We may assume $\{n-1,n\} \in E$ and that $n$ is a leaf of $G$. Hence  $G_1=G-n$ is a forest on $n-1$ vertices and by induction there exists a positive matching 
    decomposition  $E_1,\dots, E_p$ of   $G_1$ with $p=\Delta(G_1)$. 
    If $\Delta(G_1)<\Delta(G)$ we may simply set  $E_{p+1}=\{ \{n-1,n\} \}$ and note that, by virtue of   
    \ref{lem:posmat} (1),  $E_1,\dots, E_{p+1}$  is a  positive matching decomposition of $G$. If instead 
    $\Delta(G_1)=\Delta(G)$ then there exists  $i$ such that $n-1 \not\in V_{E_i}$ and hence 
    $E_i'=E_i\cup\{ \{ n-1,n\} \}$ is a matching.  Using \ref{lem:posmat} (1) and (2) one easily checks that  
    the resulting decomposition $E_1,\dots, E_{i-1},E_i', E_{i+1} \dots, E_p$ is  a  positive matching decomposition of $G$. 
  \end{itemize}
\end{proof}

Next we connect positive matching decompositions to algebraic properties of
LSS-ideals.

\begin{Lemma}
  \label{lem:termorder}
  Let $H = (V,E)$ be a hypergraph such that $E$ is a clutter, 
  $d \geq p=\pmd(H)$ and $E = \bigcup_{\ell=1}^p
  E_\ell$ a positive matching decomposition. Then there exists a term order
  $<$ on $S$ such that for every  $\ell$ and  every $A \in E_\ell$  we have
  \begin{eqnarray}
    \label{eq:termorder} 
    \ini_{<}(f_{A}^{(d)}) & = & \prod_{i \in A} y_{i\ell}.
  \end{eqnarray}
\end{Lemma}
\begin{proof}
  To define $<$ we first define weight vectors $\ww_1,\dots, \ww_p
  \in \RR^{V \times [d]}$. 
  For that purpose we use the weight functions $w_\ell : V \rightarrow \RR$,
  associated to each matching $E_\ell$, $\ell = 1,\ldots, p$.
  The weight vector $\ww_\ell$ is defined as follows: 
  \begin{itemize}
     \item $\ww_\ell(y_{ik})=0$ if $k\neq \ell$ and 
     \item $\ww_\ell(y_{i\ell})=w_\ell(i)$. 
  \end{itemize}

 By construction it follows that: 

  \begin{eqnarray}
    \label{eq:ini}
    \ini_{\ww_1}( f_{A}^{(d)}) & = & \left\{
    \begin{array}{ll}
      \prod_{i\in A} y_{i1}  & \mbox{ if } A \in E_1\\
      \displaystyle{\sum_{k=2}^d \prod_{i \in A} y_{ik}} & 
	     \mbox{ if } A \in E \setminus \{ E_1 \}.
    \end{array}
    \right.
  \end{eqnarray}

  We define the term order $<$  as follows: $y^\alpha<y^\beta$ if 
  \begin{itemize}
    \item[(1)] $|\alpha|<|\beta|$ or 
    \item[(2)] $|\alpha|=|\beta|$ and $\ww_\ell(y^\alpha)<\ww_\ell(y^\beta)$ 
      for the smallest $\ell$ such that 
      $\ww_\ell(y^\alpha)\neq \ww_\ell(y^\beta)$ or 
    \item[(3)] $|\alpha|=|\beta|$ and $\ww_\ell(y^\alpha)=\ww_\ell(y^\beta)$ 
      for all $\ell$ and $y^\alpha<_0 y^\beta$ for an arbitrary but fixed 
      term order $<_0$.  
  \end{itemize} 

  Now a simple induction  shows that
  for all $\ell$ and for all $A \in E_\ell$ we have 
  $\ini_{<}(f_{A}^{(d)})=\prod_{i \in A} y_{i\ell}$.
\end{proof}

We are ready to prove \ref{thm:pmd}:  
  
\begin{proof}[Proof of \ref{thm:pmd}] 
 Let $d\geq p = \pmd (G)$ and 
  $E=\bigcup_{\ell=1}^p E_\ell$ a $\posMa$-decomposition of $G$. 
  By \ref{lem:termorder} there is
  a term order $<$ satisfying \eqref{eq:termorder}.
  Since each $E_\ell$, $\ell =1,\ldots,p$, is a matching
  \eqref{eq:termorder} implies that the initial monomials 
  of the generators $f_A^{(d)}$ of $L_H^\KK(d)$ are 
  pairwise coprime and square free.  Then the  assertion follows from  \ref{prop:transfer}. 
 The rest follows from  \ref{thm:prime_and_ci}. 
\end{proof}  
 
The following is an immediate consequence of \ref{thm:pmd} and \ref{pmd}: 

\begin{Corollary}
   \label{cor:bounds} 
  Let $G=([n],E) $ be a graph. Then  $L_G^\KK(d)$ is a radical complete intersection for
         $d \geq \min\{ 2 n -3, |E|\}$ and prime for   $d \geq \min\{ 2 n -3, |E|\}+1$.
 If  $G$ is bipartite then    $L_G^\KK(d)$ is a radical complete intersection for
         $d \geq \min \{ n-1,|E|\}$ and prime for
         $d \geq \min \{ n-1,|E|\}+1$.
         \end{Corollary}
         
\section{Proofs of \ref{thm:prim3} and \ref{thm:forest} } 
\label{sec:6}

Next we provide the proof of \ref{thm:prim3}.

\begin{proof}[Proof of \ref{thm:prim3}] ~\\

  \begin{itemize}
    \item[(1)] 
  %We start by proving that  that $L^\KK_G(3)$ is prime if and only if $G$ does not contain $K_{1,3}$  and  $K_{2,2}$. 
  By \ref{obsKab} if $L^\KK_G(3)$ is prime then $G$ does not contain $K_{1,3}$   and $K_{2,2}$. 
  Now assume $G$ does not contain $K_{1,3}$ and $K_{2,2}$. 
  In addition, we may assume that $\KK$ is algebraically closed.  Since the tensor product over $\KK$ of $\KK$-algebras that
  are domains is a domain (see the Corollary to Proposition 1 in Bourbaki's Algebra 
  \cite[ Chapter v, 17]{Bu}) we may also assume that 
  the graph is connected.
  A connected graph not containing $K_{1,3}$ and $K_{2,2}$ is either an isolated vertex or a path 
  $P_n$ on $n>1$ vertices 
  or a cycle $C_n$ with $n$ vertices for $n =3$ or $n \geq 5$. 
  For an isolated vertex we have $L_G^\KK(3) = (0)$. 
  Hence we have to prove that $L_G^\KK(3)$ is prime when $G=P_n$ for
  $n\geq 2$ or $G=C_n$ for $n=3$ or $n\geq 5$. If $G=P_n$ then by \ref{pmd} 
  $\pmd(P_n)=\Delta(P_n)\leq 2$. 
  Hence by \ref{thm:pmd}  it follows that $L_{P_n}^\KK(3)$ is prime. 

  Now let $G=C_n$ for $n=3$ or $n \geq 5$ and 
  set $m=n-1$. 
  To prove that $L_{C_n}^\KK(3)$ is prime we use the symmetric algebra 
	perspective. Observe that $C_{n}-n$ is $P_{m} = P_{n-1}$. 
  Set $J=L_{P_{m}}^\KK(3)$, 
  $S=\KK[y_{ij} : i\in [m]  \ \ j\in [3] ]$ and $R=S/J$. We have already 
  proved that $J$ is a prime complete intersection of height $m-1$. 
  We have to prove that the symmetric algebra of the cokernel of the 
  $R$-linear map: 
 
  $$
    R^2\stackrel{Y}\to R^3 \mbox{ with } Y=
    \left(
      \begin{array}{ccc}
         y_{11}  & y_{12}  & y_{13}    \\
         y_{m1}  & y_{m2}  & y_{m3}
      \end{array}
    \right)
  $$
  is a domain. Since by \ref{Arema2} $I_2(Y)\neq 0$ in $R$, taking into 
  consideration \ref{Arema1} we may apply \ref{AHu}. Therefore, it is enough 
  to prove that $$\height\, I_1(Y)\geq 3 \mbox{ and } \height\, I_2(Y)\geq 2 
  \mbox{  in } R.$$
  Equivalently, it is enough to prove that 
  \begin{eqnarray}
    \label{eq:claim1}
	  \height\, I_1(Y)+J & \geq & m+2 \mbox{ and } \\ 
    \label{eq:claim2}
	  \height\, I_2(Y)+J & \geq & m+1 \mbox{  in } S.
  \end{eqnarray}
  First we prove \eqref{eq:claim1}.  
  Since $\height\,I_1(Y)=6$ in $S$ then \eqref{eq:claim1} is obvious for 
  $m\leq 4$. For $m>4$ observe that 
  $I_1(Y)+J$ can be written as $I_1(Y)+H$ where $H$ is the LSS-ideal of the 
  path  with vertices  $2,3,\dots, m-1$. Because $I_1(Y)$ and $H$ 
  use disjoint set of variables, we have 
  $$\height\, I_1(Y)+H=6+m-3=m+3$$
  and this proves \eqref{eq:claim1}. Now we note  that the 
  condition $\height\,I_2(Y)\geq 1$ holds in $R$ because $R$ is a domain and 
  $I_2(Y)\neq 0$. Hence we deduce from \ref{AHu}(1) that $L_{C_n}^\KK(3)$ is a 
  complete intersection for all $n\geq 3$. 

  It remains to prove \eqref{eq:claim2}. Since $I_2(Y)$ is a prime ideal of $S$ of 
  height $2$ and $J\not\subset I_2(Y)$  the ideal $I_2(Y)+J$ has height at least $3$. 
  Hence the assertion \eqref{eq:claim2} in obvious for $m=2$, i.e. $n=3$. Therefore, we may assume 
  $m\geq 4$ (here we use $n\neq 4$). Let $P$ be a prime ideal of $S$ 
  containing $I_2(Y)+J$. We have to prove that $\height\,P\geq m+1$. If $P$ 
  contains $I_1(Y)$ then $\height\,P\geq  m+2$ by \eqref{eq:claim1}. So we 
  may assume that $P$ does not contain $I_1(Y)$, say $y_{11}\not\in  P$, and 
  prove that $\height\,PS_{x}\geq m+1$ where $x=y_{11}$. 
  Since $I_2(Y)S_{x}=(y_{m2}-x^{-1}y_{m1}y_{12}, y_{m3}-x^{-1}y_{m1}y_{13} )$  
  we have

  \begin{eqnarray*}
	  f_{m-1,m}^{(3)} & = & y_{m-1,1}y_{m1}+y_{m-1,2}y_{m2}+y_{m-1,3}y_{m3} \\
			  & = & y_{m-1,1}y_{m1}+y_{m-1,2}x^{-1}y_{m1}y_{12}+y_{m-1,3}x^{-1}y_{m1}y_{13} \\ 
			  & = & x^{-1}y_{m1} f_{1,m-1}^{(3)} \mod I_2(Y)S_{x}
  \end{eqnarray*}

  From $f_{m-1,m}^{(3)}\in J$ it follows that $y_{m1} f_{1,m-1}^{(3)} \in PS_x$. 
  This implies that either $y_{m1}\in PS_x$ or $f_{1,m-1}^{(3)}\in PS_x$. 
  In the first case $PS_x$ contains $y_{m1},y_{m2}, y_{m3}$ and the LSS-ideal 
  associated to the path with vertices $1,\dots, m-1$. Hence 
  $\height\,PS_x\geq 3+m-2=m+1$ as desired. 
  Finally, if $f_{1,m-1}^{(3)}\in PS_x$ we have that $PS_x$ contains the 
  ideal $L_{C_{m-1}}^\KK(3)$ associated to the  cycle with vertices 
  $1,\dots, m-1$ and we have already observed that this ideal is a complete 
  intersection. Since $y_{m2}-x^{-1}y_{m1}y_{12}, y_{m3}-x^{-1}y_{m1}y_{13}$ 
  are in $PS_x$ as well it follows that $\height\,PS_x\geq 2+m-1=m+1$. 
 
  \item[(2)] 
%Now we prove the second part of the statement:  $L^\KK_G(2)$ is a complete intersection  if and only if $G$ does not contain $K_{1,3}$  and  $C_{2m}$ with $m\geq 2$. 
  
    For the ``only if'' part we note that if $L^\KK_G(2)$ is a complete intersection 
    then $L^\KK_G(3)$ is prime by \ref{thm:prime_and_ci} and hence $G$ does cannot contain 
    $K_{1,3}$  by \ref{obsKab}. Suppose, by contradiction, that $G$ contains $C_{2m}$ for some 
    $m \geq 2$. Hence $L_{C_{2m}}^\KK(2)$ is a complete intersection of height $2m$. But the 
    generators of $L_{C_{2m}}^\KK(2)$ are (up to sign) among the $2$-minors of 
    the matrix: 
    $$\left (
      \begin{array}{cccccc}
        y_{11} & -y_{22}  &    y_{31} &\dots & y_{2m-1,1} &  -y_{2m,2} \\
        y_{12} &  y_{21}  &    y_{32} &\dots & y_{2m-1,2} &   y_{2m,1} 
      \end{array} 
      \right )
    $$
    and the ideal of $2$-minors of such a matrix has height $2m-1$,  
    a contradiction. 
  
    For the converse implication, we may assume that $\KK$ is algebraically closed. Since the 
    tensor product over a perfect field $\KK$ of reduced $\KK$-algebras is 
    reduced \cite[ Thm 3, Chapter V, ¤15]{Bu}, we may also assume that $G$ is 
    connected. A connected graph satisfying the assumptions is either  an isolated vertex,  
    or a path or a cycle with a odd number of vertices. 
    We have already observed that $\pmd(P_n)=\Delta(P_n)\leq 2$. 
    By \ref{thm:pmd} it follows that $L_{P_n}^\KK(2)$ is a complete intersection. 
    It remains to prove that $L^\KK_{C_{2m+1}}(2)$ is a complete intersection 
    (of height $2m+1$). Note that $L^\KK_{P_{2m+1}}(2)\subset 
    L^\KK_{C_{2m+1}}(2)$ and we know already that $L^\KK_{P_{2m+1}}(2)$ is a 
    complete intersection of height $2m$. Hence it remains to prove that 
    $f_{1,2m+1}^{(2)}$ does not belong to any  minimal prime of 
    $L^\KK_{P_{2n+1}}(2)$. 
    The generators of $L^\KK_{P_{2n+1}}(2)$ are (up to sign) the adjacent 
    $2$-minors of the matrix: 
 
    $$
      Y=\left (
      \begin{array}{cccccccc}
        y_{11} & -y_{22}  &    y_{31} &\dots & y_{2m-1,1} &  -y_{2m,2}  &  y_{2m+1,1}\\
        y_{12} &  y_{21}  &    y_{32} &\dots & y_{2m-1,2} &   y_{2m,1}   &  y_{2m+1,2}
      \end{array} 
      \right )
    $$

    The minimal primes of $L^\KK_{P_{2n+1}}(2)$ are described in the proof of 
    \cite[Thm.4.3]{DES}, see also \cite{HS} and \cite{HHHKR}. By the description 
    given in \cite{DES} it is easy to see that all minimal primes of 
    $L^\KK_{P_{2n+1}}(2)$ with the exception of $I_2(Y)$ are contained in the 
    ideal $Q=(y_{ij} : 2<i<2m+1 \ \ 1\leq j\leq 2)$.  
    Clearly, $f_{1,2m+1}^{(2)}\not\in Q$. Finally, one has 
    $f_{1,2m+1}^{(2)}\not\in I_2(Y)$ since the monomial $y_{11}y_{2m+1,1}$ 
    is divisible by no monomials in the support of the generators of $I_2(Y)$. 
  \end{itemize}
\end{proof}

We proceed with the proof of \ref{thm:forest}.  We 
first formulate a more general statement. 
For this we need to introduce the concept of Sturmfels-Cartwright 
ideals. This concept was coined in \cite{CDG2} 
inspired by earlier work in \cite{CDG1} and \cite{CS}. It was further 
developed and applied to various classes of ideals in \cite{CDG3} and 
\cite{CDG4}. 

Consider for $d_1,\ldots, d_n \geq 1$ the polynomial ring 
$S=\KK[y_{ij} : i \in [n], j \in [d_i]]$ with multigrading
$\deg y_{ij}=\ee_i\in \ZZ^n$. The group 
$G=\GL_{d_1}(\KK)\times \dots \times \GL_{d_n}(\KK)$ acts naturally on $S$ 
as the group of $\ZZ^n$-graded $K$-algebra automorphism. The Borel subgroup 
of $G$ is $B=U_{d_1}(\KK)\times \dots \times U_{d_n}(\KK)$ where $U_d(\KK)$ 
denotes the subgroup of upper triangular matrices in $\GL_d(\KK)$. 
A $\ZZ^n$-graded ideal $J$ is Borel fixed if $g(J)=J$ for every 
$g\in B$. A $\ZZ^n$-graded ideal $I$ of $S$ is
called a Cartwright-Sturmfels ideal if there exists a radical Borel
fixed ideal $J$ with the same multigraded Hilbert-series. 

\begin{Theorem}
  \label{CaSt}
  For $d_1,\ldots, d_n \geq 1$ let 
  $S=\KK[y_{ij} : i \in [n], j \in [d_i]]$ be the polynomial ring with 
  $\ZZ^n$ multigrading induced by $\deg y_{ij}=\ee_i\in \ZZ^n$  
  and $G = (V,E)$ be a forest. For each 
  $e = \{i,j\} \in E$ let $f_e \in S$ be a $\ZZ^n$-graded polynomial  
  of degree $\ee_i + \ee_j$. 
  Then $I = (f_e~:~e\in E)$ is a Cartwright-Sturmfels ideal.   
  In particular, $I$ and all its initial ideals are radical. 
\end{Theorem} 

\begin{proof} 
  First, we observe that we may assume that the generators $f_e$ of $I$ form a 
  regular sequence. To this end we introduce new variables and for each $e=\{i,j\}\in E$ we add to 
  $f_e$ a monomial $m_e$ in the new variables of degree $e_i+e_j$ so that $m_e$ and 
  $m_{e'}$ are coprime if $e \neq e'$. The new polynomials $f_e+m_e$ with 
  $e\in E$ form a regular sequence by \ref{prop:transfer} since their initial 
  terms with respect to an appropriate term order are the pairwise coprime  
  monomials $m_e$.  
  The ideal $I$ arises as a multigraded linear section of the ideal 
  $(f_e+m_e : e\in E)$ by setting all new variables to $0$. 
  By \cite[Thm. 1.16(5)]{CDG1} the family of Cartwright-Sturmfels ideals 
  is closed under any multigraded linear section. Hence it is enough to prove 
  the statement for the ideal $(f_e+m_e : e\in E)$. Equivalently we may assume 
  right away  that  the generators $f_e$ of $I$ form a regular sequences. 
  
  The multigraded Hilbert series of a multigraded $S$-module $M$ can by written 
  as $$\frac{K_M(z_1,\dots,z_n)}{ \prod_{i=1}^n (1-z_i)^{d_i} }.$$
  The numerator $K_M(z_1,\dots,z_n)$ is a Laurent polynomial polynomial with 
  integral coefficients called the $K$-polynomial of $M$. 
  Since the $f_e$'s form a regular sequence the $K$-polynomial of $S/I$ is 
  the polynomial: 
  $$
    F(z)=F(z_1,\dots,z_n)=\prod_{\{i,j\}\in E} (1-z_iz_j)
         \in \QQ[z_1,\dots, z_n].
  $$ 
  
  To prove that $I$ is Cartwright-Sturmfels we have to prove that there is a   
  Borel-fixed radical ideal $J$ such that the $K$-polynomial of $S/J$ is $F(z)$. 
  Taking into consideration the duality between Cartwright-Sturmfels ideals 
  and Cartwright-Sturmfels$^*$ ideals discussed in \cite{CDG2}, it is enough to 
  exhibit a monomial ideal $J$ whose generators are in the polynomial ring 
  $S'=\KK[y_{1},y_{2},\dots, y_{n}]$ equipped with the (fine) $\ZZ^n$-grading 
  $\deg y_i=\ee_i \in \ZZ^n$ such that the $K$-polynomial of $J$ regarded as an 
  $S'$-module is $F(1-z_1,\dots, 1-z_n)$, that is, 
  $$\prod_{\{i,j\}\in E} (z_i+z_j-z_iz_j).$$
  We claim that, under the assumption that $([n], E)$ is a forest, the ideal 
  $$J=\prod_{\{i,j\}\in E} (y_{i}, y_{j})$$ 
  has the desired property. In other words, we have to prove that the tensor 
  product  
  $$T_E=\bigotimes_{\{i,j\} \in E}  T_{\{i,j\}}$$ of the truncated Koszul 
  complexes: 
  $$T_{\{i,j\}}: 0\to S'(-\ee_i-\ee_j)  \to S'(-\ee_i)\oplus  S'(-\ee_j)\to 0$$
  associated to $y_{i}, y_{j}$ resolves the ideal $J$. 
  Consider a leaf $\{a,b\}$ of $E$.  Set $E'=E\setminus \{ \{a,b\} \}$, 
  $$J'=\prod_{\{i,j\}\in E'} (y_{i}, y_{j})$$ 
  and $J''=(y_a,y_b)$.  
  Then by induction on the number of edges we have that $T_{E'}$ resolves 
  the ideal $J'$. Then the homology of $T_E$ is $\Tor^{S'}_*(J',J'')$. 
  Since $\{a,b\}$ is a leaf, one of the two variables $y_a,y_b$ does not appear at all in the 
  generators of $J'$. Hence $y_a,y_b$ forms a regular $J'$-sequence.  
  Then $\Tor^{S'}_{\geq 1}(J',J'')=0$ and  hence $T_E$ resolves $J'\otimes J''$. 
  Finally, $J'\otimes J''=J'J''$ since $\Tor^{S'}_{1}(J',S/J'')=0$. 
  This concludes the proof that the ideal $I$ is a Cartwright-Sturmfels ideal. 
  Every initial ideal of a Cartwright-Sturmfels ideal is a Cartwright-Sturmfels 
  ideal as well because this property just depends on the Hilbert series. 
  In particular, every initial ideal of a Cartwright-Sturmfels ideal is 
  radical. 
\end{proof} 
  
Now we are ready to prove \ref{thm:forest}: 
  
\begin{proof}[Proof of \ref{thm:forest}] 
        Setting $d_1= \cdots = d_n = d$ and $f_e = f_e^{(d)}$ in \ref{CaSt}
      we have that $L_G^\KK(d)$ is a Cartwright-Sturmfels ideal and hence 
      radical. Assertions (2) and (3) follows  from \ref{pmd},  \ref{thm:pmd}, \ref{obsKab} and \ref{thm:prime_and_ci}. 
\end{proof}

\section{Invariant theory, determinantal ideals of matrices with $0$'s and their 
relation to LSS-ideals } 
\label{sec:minorgraph}

The first goal of this section is to recall some classical results from invariant 
theory, see for example the paper by De Concini and Procesi \cite{DP}. In particular, we recall how  
determinantal/Pfaffian rings arise as invariant rings of group actions.  
We assume throughout this section that the base field $\KK$ is of characteristic $0$.
After the recap of invariant theory we will establish the connection to LSS-ideals. 

\subsection{Generic determinantal rings as rings of invariants (\generic)} 

We take an $m\times n$ matrix of variables 
$X^\generic_{m,n}=(x_{ij})$ 
and consider the ideal $I_{d+1}^\KK(X^\generic_{m,n})$ of $S^\generic=\KK[x_{ij}~:~(i,j) \in [m] \times [n]]$ 
generated by the $(d+1)$-minors of $X_{m,n}^\generic$. 
Consider two matrices of variables $Y$ and $Z$ of size 
$m\times d$ and $d\times n$ and the following 
action of $\group=\GL_d(\KK)$ on the polynomial ring $\KK[Y,Z]$: 
The matrix $A\in \group$ acts by the $\KK$-algebra automorphism of   
$\KK[Y,Z]$ that sends $Y\to YA$ and $Z\to A^{-1}Z$. 
The entries of the product matrix $YZ$ are clearly invariant under this action.  
Hence the ring of invariants $\KK[Y,Z]^\group$ contains the subalgebra 
$\KK[YZ]$ generated by the entries of the product $YZ$. 
The {\it First Main Theorem of Invariant Theory} for this 
action says that $\KK[Y,Z]^\group=\KK[YZ]$.
We have a surjective $\KK$-algebra map: 

$$\phi: S^\generic \to \KK[Y,Z]^\group=\KK[YZ]$$ 

sending $X$ to $YZ$. 
Clearly the product matrix $YZ$ has rank 
$d$ and hence we have $I_{d+1}^\KK(X_{m,n}^\generic)\subseteq \Ker \phi$. 
The {\it Second  Main Theorem of Invariant Theory}
says that $I_{d+1}^\KK(X_{m,n}^\generic)=\Ker \phi$. Hence 

\begin{eqnarray}
  \label{eq:isogen}
  S/I_{d+1}^\KK(X_{m,n}^\generic) & \simeq &  \KK[YZ]
\end{eqnarray}
   
\subsection{Generic symmetric determinantal rings as rings of invariants (\symmetric)} 
	    
We take an $n\times n$ symmetric matrix of variables $X_n^\symmetric=(x_{ij})$  and consider the ideal 
$I_{d+1}^\KK(X_n^\symmetric)$ in $S^\symmetric=\KK[x_{ij}~:~1 \leq i \leq j \leq n]$ generated by the 
$(d+1)$-minors of $X_n^\symmetric$. 
Consider a  matrix of variables $Y$  
of size $n\times d$ and the following action of the orthogonal group 
$\group=\OO_d(\KK)=\{ A \in \GL_d(\KK) : A^{-1}=A^T\}$ 
on the polynomial ring $\KK[Y]$:
Any $A\in \group$ acts by the $\KK$-algebra automorphism of  
$\KK[Y]$ that sends  $Y$ to $YA$. 
The entries of the product matrix $YY^T$ are invariant under 
this action and hence the ring of invariants contains  
the subalgebra $\KK[YY^T]$ generated by the entries of $YY^T$. 
The {\it First Main Theorem of Invariant Theory} for this action 
asserts  that $\KK[Y]^G=\KK[YY^T]$.
Then we have a surjective presentation: 
$$\phi:S^\symmetric \to  \KK[YY^T]$$ 
sending $X$ to $YY^T$. Since  the product matrix $YY^T$ has  rank $d$   
we have $I_{d+1}(X)\subseteq \Ker \phi$.  
The {\it Second Main Theorem of Invariant Theory} then 
says that $I_{d+1}(X)=\Ker \phi$. Hence 

\begin{eqnarray}
  \label{eq:isosym}
  S^\symmetric/I_{d+1}^\KK(X_n^\symmetric) & \simeq &  \KK[YY^T].
\end{eqnarray}

\subsection{Generic Pfaffian rings as rings of invariants (\skewsymmetric)}

We take an $n\times n$ skew-symmetric matrix of variables $X_n^\skewsymmetric=(x_{ij})$ 
 and consider the ideal $\Pf^\KK_{2d+2}(X)$ generated by the Pfaffians 
of size $(2d+2)$   of
$X_n^\skewsymmetric$ in  
$S^\skewsymmetric=\KK[x_{ij}~:~1 \leq i < j \leq n]$. 
Consider a matrix of variables $Y$  
of size $n\times 2d$ and let $J$ 
be the $2d\times 2d$ block matrix with $d$ blocks 
    
$$\left(
      \begin{array}{cc}
        0 & 1\\
        -1& 0
      \end{array}
      \right)
$$
on the diagonal and $0$ in the other positions. The 
sympletic group $\group=\Sp_{2d}(\KK)=\{ A \in \GL_{2t}(\KK) :  AJA^T=J\}$ 
acts on the polynomial ring $\KK[Y]$ as follows: 
an $A\in \group$ acts on $\KK[Y]$ by the automorphism that
sends $Y\to YA$. The entries of the product matrix $YJY^T$ 
are invariant under this action and hence the ring of invariants 
contains the subalgebra $\KK[YJY^T]$ generated by the entries of 
$YJY^T$.  The {\it First Main Theorem of Invariant Theory} 
for the current action says that 
$\KK[Y]^G=\KK[YJY^T]$.
Then we have a surjective presentation: 
$\phi:S^\skewsymmetric  \to  \KK[YY^T]$ 
sending $X$ to $YJY^T$. The product matrix $YJY^T$ has rank $2d$ and hence 
we have $\Pf^\KK_{2d+2}(X)\subseteq \Ker \phi$.  
The {\it Second Main Theorem of Invariant Theory} for this action 
says that $\Pf^\KK_{2d+2}(X) =\Ker \phi$. Hence 

\begin{eqnarray}
  \label{eq:isoskew}
  S^\skewsymmetric/\Pf^\KK_{2d+2}(X_n^\skewsymmetric) & \simeq & \KK[YJY^T].
\end{eqnarray}
   
 \subsection{Determinantal ideals of matrices with $0$'s and their 
relation to LSS-ideals }

The classical invariant theory point of view   shows that the generic determinantal and Pfaffian  
ideals are prime as they are kernels of ring maps whose codomains are 
integral domains. Their height is also well known (see for example 
\cite{BV} and its list of references): 
  
\begin{itemize}
  \item[(\generic)] The height of the ideal $I_d^\KK(X_{m,n}^\generic)$ 
     of $d$-minors of a $m\times n$  
     matrix of variables is $(n+1-d)(m+1-d)$. 
  \item[(\symmetric)] The height of the ideal $I_d^\KK(X_n^\symmetric)$ of $d$-minors of a symmetric 
	  $n\times n$ matrix of variables is ${n-d+2 \choose 2}$. 
  \item[(\skewsymmetric)] The height  of the ideal of Pfaffians $\Pf^\KK_{2d}(X_n^\skewsymmetric)$ of size $2d$ 
     (and degree $d$) of an $n\times n$ skew-symmetric matrix of variables 
		is ${n-2d+2 \choose 2}$. 
\end{itemize} 
 
If one replaces the entries of the matrices with general linear forms in, 
say, $u$ variables, then Bertini's theorem in combination with the fact that 
the generic determinantal/Pfaffian rings are Cohen-Macaulay implies that  
the determinantal/Pfaffian ideals remain prime as long as 
$u\geq 2$+height and radical if $u\geq 1$+height.   

But what about the case of special linear sections of determinantal ideals of 
matrices? And what about the case of coordinate sections? 
Are the corresponding ideals prime or radical? 
To describe coordinate sections we employ the following notation. 

\begin{itemize}
  \item[(\generic)]
   In the generic case we take  
   a bipartite graph $G = ([m] \cup [\tilde{n}],E)$    and 
   denote by $X_G^\generic$ the matrix obtained from the $m\times n$ matrix of 
   variables  by replacing 	the entries in position $(i,j)$ with $0$ for  all $\{i,\tilde{j}\} \in  E$.
  \item[(\symmetric)]
   In the generic symmetric case   we take 
   a subgraph $G = ([n],E)$ of $K_n$ and denote by $X_G^\symmetric$ 
   the matrix obtained from the  $n\times n$  symmetric matrix  of variables 
   by replacing  with $0$ the entries in position $(i,j)$ and $(j,i)$  for all $\{i,j\} \in E$.   
  \item[(\skewsymmetric)]
   In the generic skew-symmetric case we take a subgraph 
   $G = ([n],E)$ of $K_n$ and denote by $X_G^\skewsymmetric$ 
   the matrix obtained from the skew-symmetric matrix  of variables  
     by replacing  with $0$  the entries in position $(i,j)$ and $(j,i)$  
     for all $\{i,j\} \in E$.  
     \end{itemize}

In this terminology $I_d^\KK(X_G^\generic)$ is the ideal of $d$-minors of $X_G^\generic$  
in $S^\generic$  and similarly in the symmetric case. 
We write $\Pf^\KK_{2d}(X_G^\skewsymmetric)$ for the ideal of Pfaffians of 
size $2d$ of $X_G^\skewsymmetric$ in $S^\skewsymmetric$. 
We ask for conditions on $G$ that imply that
$I_d^\KK(X_G^\generic)$ (resp. $I_d^\KK(X_G^\symmetric)$, resp. $\Pf^\KK_{2d}(X_G^\skewsymmetric)$) is
radical or prime or has the expected height.  

Clearly, special linear sections of generic determinantal ideals can give non-prime and non-radical 
ideals. On the positive side, for maximal minors, we have the following  
results: 

\begin{Remark}
  \begin{itemize}
    \item[(1)] Eisenbud \cite{E} proved that the ideal of maximal minors of a 
      $1$-generic $m\times n$ matrix of linear forms is prime and remains 
      prime  even  after modding out any set of $\leq m-2$ linear forms.  
      In particular, the ideal of maximal minors of an $m\times n$ matrix of 
      linear forms is prime provided the ideal generated by the entries of 
      the matrix has at least $m(n-1)+2$ generators. 
    \item[(2)] Giusti and Merle in \cite{GM} studied the ideal of maximal 
      minors of coordinate sections in the generic case. One of their main 
      results, \cite[Thm.1.6.1]{GM} characterizes, in combinatorial terms,  
      the subgraphs $G$ of $K_{m,n}$, $m \leq n$, such that 
      the variety associated to $I_m^\KK(X_G^\generic)$ is irreducible, 
      i.e. the radical of $I_m^\KK(X_G^\generic)$ is prime. 
    \item[(3)] Boocher proved in \cite{Bo} that  
      for any subgraph $G$ of $K_{m,n}$, $m \leq n$, the ideal 
      $I_m^\KK(X_G^\generic)$ is radical. Combining his result with the result 
      of 
      Giusti and Merle, one obtains a characterization of the graphs $G$ 
      such that $I_m^\KK(X_G^\generic)$ is prime. 
    \item[(4)] Generalizing the result of Boocher, in \cite{CDG1} and 
      \cite{CDG2} it is proved that ideals of maximal minors of a matrix of 
      linear forms that is either row or column multigraded is radical.  
  \end{itemize}
\end{Remark} 
  
In the generic case every non-zero minor of a matrix of type 
$X_G^\generic$ has no multiple factors because its multidegree is square-free. 
This explains, at least partially,  why  the determinantal ideals of $X_G^\generic$  have the tendency to be radical. 
However, the following example shows that  they are not radical in general. 
  
\begin{Example}
  \label{ex:4mnotrad} 
  Let $X_G^\generic$ be the $6\times 6$ matrix associated to the graph
  from \ref{ex:nrad}(3). That is, in the $6\times 6$ generic matrix we set to $0$ the entries in positions  
  $$(1,1), (1,2), (1,3), (1,4), (2,1), (2,2), (3,2), (3,3),  (4,3), (4,4), (5,1), (5,4).$$
  Then $I_4^\KK(X_G^\generic)$ is not radical over a field of characteristic $0$ 
  and very likely over any field.  Here the ``witness'' is $g =x_{1,5}$, i.e.  $I_4^\KK(X_G^\generic):g\neq I_4^\KK(X_G^\generic):g^2$. 
 Since $G$ is contained in $K_{5,4}$ one can consider as well $I_4^\KK(X_G^\generic)$ in the $5\times 5$ matrix but that ideal turns out to be radical.
 \end{Example}   
 
Similarly for symmetric  matrices we have: 

\begin{Example}
  \label{ex:5symnotrad} 
  Let $X_G^\symmetric$ be the $7\times 7$ generic symmetric matrix associated 
  to the graph
  from \ref{ex:nrad}(1). That is, in the $7\times 7$ generic symmetric matrix we  set to $0$ the entries in positions  
  $$\{1, 2\}, \{1, 3\}, \{1, 4\}, \{1, 5\}, \{2, 3\}, \{2, 4\}, \{2, 6\}, 
	 \{3, 5\}, \{4, 6\}$$
  as well as in the symmetric positions. Then $I_4^\KK(X_G^\symmetric)$ is not 
  radical over a field of characteristic $0$. 
  The witness here is $g=x_{1,6}$. Since $G$ is contained in $K_{6}$ one can 
  consider as well $I_4^\KK(X_G^\symmetric)$ in the $6\times 6$ matrix but that 
  ideal turns out to be radical.
 \end{Example}  
  
It turns out that \ref{ex:nrad}, \ref{ex:4mnotrad} and \ref{ex:5symnotrad} 
are indeed closely related as we now explain.    

Let $G = ([m] \cup [\tilde{n}],E)$ be a subgraph of the complete bipartite 
graph $K_{m,n}$.  
In view of the isomorphism  \eqref{eq:isogen} we have that 
$$S^\generic/\Big( I_{d+1}^\KK(X_{m,n}^\generic) + (x_{ij}~|~\{i,\tilde{j}\} \in E) \Big) \simeq \KK[YZ]/J_G(d)$$ 
where $Y=(y_{ij}),Z=(z_{ij})$ are respectively $m\times d$ and $d\times n$ 
matrices of variables and  $J_G(d)$ is the ideal of $\KK[YZ]$ generated by 
$(YZ)_{i,j}$ with $\{i,\tilde{j}\} \in E$. Furthermore 
$$I_{d+1}^\KK(X_{m,n}^\generic) + (x_{ij}~|~\{i,\tilde{j}\} \in E) = I_{d+1}^\KK(X_G^\generic) + (x_{ij}~|~\{i,\tilde{j}\} \in E).$$

The LSS-ideal 
$L_G^\KK(d)\subset \KK[Y,Z]$ is indeed equal to $J_G(d)\KK[Y,Z]$. Now it is a 
classical  result in  invariant 
theory (derived from the fact that linear groups are reductive in characteristic $0$),
that $\KK[YZ]$ is a direct summand of $\KK[Y,Z]$ in characteristic $0$. 
This implies that 

\begin{eqnarray*}
  J_G(d) & = & L_G^\KK(d) \cap \KK[YZ].
\end{eqnarray*}
 
The next proposition is an immediate consequence. 
 
\begin{Proposition} 
  \label{prop:traG}
  Let $\KK$ be a field of characteristic $0$,  $d \geq 1$ and 
  $G= ([m] \cup [\tilde{n}],E)$ 
  be a subgraph of $K_{m,n}$. 
  If  $L_G^\KK(d)$ is radical (resp.~is a complete 
  intersection, resp.~is prime)
  then  $I_{d+1}^\KK(X_G^\generic)$ is radical (resp.~has maximal height, resp.~is  prime). 
\end{Proposition} 
 
Now we start from a subgraph $G$ of $K_n$.
For $d+1\leq n$ we may consider the coordinate 
section $I_{d+1}^\KK(X_G^\symmetric)$ of $I_{d+1}^\KK(X^\symmetric_n)$. 
Using the isomorphism \eqref{eq:isosym} we obtain: 
 
\begin{Proposition} 
  \label{traS}
  Let $\KK$ be a field of characteristic $0$ and 
  $G = ([n],E)$ a graph.
  If  $L_G^\KK(d)$ is radical  (resp.~is a complete
  intersection, resp.~is prime) 
  then   $I_{d+1}^\KK(X_G^\symmetric)$   is radical (resp.~has maximal height, resp.~is  prime). 
\end{Proposition}

For $2d+2\leq n$ we may consider the coordinate section 
$\Pf^\KK_{2d+2}(X_G^\skewsymmetric)$ of $\Pf^\KK_{2d+2}(X^\skewsymmetric_n)$. 
We may as well consider the associated twisted 
LSS-ideal $\hat L_G^\KK(d)$ defined as follows. 
For every $i\in [n]$  we consider $2d$ indeterminates 
$y_{i\,1},\dots, y_{i\,2d}$. For $e = \{i,j\}$, $1\leq i<j\leq n$ we set  
$\hat{f}_{e}^{(d)}$ to be the entry of the matrix $YJY^T$ in row $i$ and
column $j$, i.e. 

$$\hat{f}_e^{(d)}=\sum_{k=1}^d \, \Big(y_{i\,2k-1}y_{j\,2k}-y_{i\,2k}y_{j\,2k-1} \Big).$$

Then we set 
 
$$\hat L_G^\KK(d)=( \hat{f}_e^{(d)}~  :~ e \in E ).$$

the twisted LSS-ideal associated to $G$.
For $d=1$ the twisted LSS-ideal coincides with the so-called binomial edge 
ideal defined and studied in \cite{HHHKR,KM,MM,O}. 

\begin{Remark} 
  Assume $G$ is bipartite with bipartition 
  $[n] = V_1 \cup V_2$ then the coordinate transformation (see \cite[Cor. 6.2]{BMS}) 

  \begin{itemize}
    \item $y_{i\,2k-1} \mapsto y_{i\,2k-1}$ and 
	  $y_{i\,2k} \mapsto y_{i\,2k}$ 
          for $i \in V_1$,
    \item $y_{j\,2k} \mapsto  y_{j\,2k-1}$ and
          $y_{j\,2k-1} \mapsto  - y_{j\,2k}$ 
          for $j \in V_2$,
  \end{itemize}

  sends $\hat L_G^\KK(d)$ to $L_G^\KK(2d)$. In particular, for a bipartite 
  graph $G$ we have that $\hat L_G^\KK(d)$ is 
  radical  (resp.~prime) if and only if $L_G^\KK(2d)$
 is  radical  (resp.~prime).
\end{Remark} 

Using the isomorphism \eqref{eq:isoskew} we obtain: 
   
\begin{Proposition} 
  \label{traA}
  Let $\KK$ be a field of characteristic $0$ and 
  $G=([n], E)$ a graph.
  If  $\hat L_G^\KK(d)$ is radical (resp.~is a complete intersection, resp.~is  prime) 
  then  $\Pf^\KK_{2d+2}(X_G^\skewsymmetric)$ is radical (resp.~has maximal height, resp.~is  prime). 
\end{Proposition} 
  
Now, in characteristic $0$, the results that we have established for LSS-ideals can be turned 
into statements concerning coordinate sections of determinantal ideals.   

\begin{Theorem}
  \label{pro}
  Let $\KK$ be a field of characteristic $0$. 
  \begin{itemize} 
     \item[(1)] For every subgraph $G$  of $K_{m,n}$  
      the ideals $I_2^\KK(X_G^\generic)$ and $I_3^\KK(X_G^\generic)$ are radical. 
     \item[(2)] For every subgraph $G$  of  $K_n$   the ideals 
       $I_2^\KK(X_G^\symmetric)$ and $I_3^\KK(X_G^\symmetric)$ are radical. 
     \item[(3)] For every subgraph $G$  of  $K_n$  the ideal 
       $\Pf^\KK_4(X_G^\skewsymmetric)$ is radical. 
  \end{itemize} 
  Furthermore if $G$ is a forest then  
  \begin{itemize} 
     \item[(4)] $I_d^\KK(X_G^\generic), I_d^\KK(X_G^\symmetric)$ and 
       $\Pf^\KK_{2d}(X_G^\skewsymmetric)$  are  radical for all $d$. 
     \item[(5)] $I_d^\KK(X_G^\generic)$ and $I_d^\KK(X_G^\symmetric)$ have 
       maximal height if $d\geq \Delta(G)+1$. 
     \item[(6)] $I_d^\KK(X_G^\generic)$ and $I_d^\KK(X_G^\symmetric)$ are
       prime if  $d\geq \Delta(G)+2$. 
  \end{itemize}
\end{Theorem}
  
\begin{proof} 
  The statements for  ideals of $2$-minors   follow from \ref{prop:traG} and \ref{traS} using the fact 
  that the edge ideal of a graph is radical. Indeed  these results hold over 
  a field of arbitrary characteristic as the corresponding ideals are 
  ``toric.'' 

  By \cite[Thm. 1.1]{HMMW} the ideal $L_G^\KK(2)$ is radical for all graphs 
  $G$. Using \ref{prop:traG} and \ref{traS} this implies that 
  $I_3^\KK(X_G^\generic)$ is radical for bipartite graphs $G$ and 
  $I_3^\KK(X_G^\symmetric)$ is radical for  arbitrary graphs.
    
  By \cite[Cor. 2.2]{HHHKR} the ideal $\hat L_G^\KK(1)$ is radical for all 
  graphs $G$. Using \ref{traA} this implies that 
  $\Pf^\KK_4(X_G^\skewsymmetric)$ is radical for arbitrary graphs.
  
  Finally, for a forest $G$ the results in the case of minors are derived 
  from \ref{prop:traG}, \ref{traS} and \ref{thm:forest}. In the Pfaffian case
  they follow using \ref{CaSt} and \ref{traA}.
\end{proof}

The following corollary is an immediate consequence of assertion  
(3) in \ref{pro}.

\begin{Corollary}
  \label{cor:grass}
  Let $G(2,n)$ be the coordinate ring of the Grassmannian of       
  $2$-dimensional subspaces in $\KK^n$ in its standard Pl\"ucker 
  coordinates. Then any subset of the Pl\"ucker coordinates generates a 
  radical ideal in $G(2,n)$. 
\end{Corollary}

A statement analogous to \ref{cor:grass} for higher order Grassmannians is 
not true. Indeed, the point is that a set of $m$-minors of a generic matrix 
$m\times n$ does not generate a radical ideal in general (as it does for 
$m=2$). For example, in the Grassmannian $G(3,6)$ modulo $[123], [124], 
[135], [236]$ the class of $[125][346]$ is a non-zero nilpotent.   
  
Next we look into necessary conditions for $I_d^\KK(X_G^\generic)$ and 
$I_d^\KK(X_G^\symmetric)$ to be prime.

\begin{Lemma}
  \label{lem:minorKab}
  Let $G = ([n],G)$ be a graph. 
  \begin{itemize}
    \item[(1)] If $I_{d+1}^\KK(X_G^\symmetric)$ is prime then $G$ does not  
      contain $K_{a,b}$ for $a+b > d$ (i.e. $\bar{G}$ is $(n-d)$-connected).

     \item[(2)] If $G=B_d$ with $d\geq 4$ and $X$ is the generic $(d+2)\times (d+2)$ 
      matrix then $I_{d+1}^\KK(X_G^\generic)$  is not prime. 
  \end{itemize}
\end{Lemma}

\begin{proof} 
  \begin{itemize}
     \item[(1)] 
        Assume by contradiction that $G$ contains $K_{a,b}$ for $a+b=d+1$. We may assume 
        that the corresponding set of vertices are $[a]$ and $\{a+j : j \in [b] \}$.  
        But then the submatrix of $X_G^\symmetric$ of the first $d+1$ rows and columns  
        is block-diagonal with (at least) two blocks. 
        Hence its determinant is non-zero, is reducible and has degree $d+1$.  Since all the 
        generators of $I_{d+1}^\KK(X_G^\symmetric)$ have degree $d+1$  it follows that 
        $I_{d+1}^\KK(X_G^\symmetric)$ cannot be prime.

     \item[(2)] Set $Y_d=X_{B_d}^{\generic}$, i.e.,   $$Y_d= 
          \left( 
          \begin{array}{cccccccc}
           x_{11}    & 0      & \cdots & 0     & x_{1,d+1} &  x_{1,d+2} \\ 
                0    & x_{22}   & \cdots & 0        & x_{2,d+1} &  x_{2,d+2} \\ 
             \vdots  & \vdots & \vdots & \vdots & \vdots  & \vdots\\ 
                0    & \cdots &      0 & x_{dd} &   \vdots & \vdots \\
          x_{d+1,1}  &  x_{d+1,2} &  \cdots &  \cdots  & x_{d+1,d+1}  & x_{d+1,d+2} \\
          x_{d+2,1}  &  x_{d+2,2} &  \cdots  &   \cdots & x_{d+2,d+1}  &  x_{d+2,d+2}
          \end{array} \right).
        $$ 
        and $J=I_{d+1}(Y_d)$ and let $S$ be the polynomial ring whose 
	indeterminates are the non-zero entries of $Y_d$. 
        First, we prove that for every $d\geq 1$ the ideal $J$ has the 
        expected height, i.e. $\height\,J=4$. For $d=1,2,3$ the ideal $J$  
        is indeed prime of height $4$: for $d=1$ this is obvious because 
        $Y_1$ is the generic $3\times 3$ matrix; for $d=2$ and $d=3$ it 
        follows from the fact that the corresponding LSS-ideal is prime by 
        virtue of \ref{prop:traG}. 
        For $d>3$ let $P$ be a prime containing $J$. If $P$ contains 
        $(x_{11},x_{22},x_{33},x_{44})$ then $\height\,P\geq 4$. If $P$ does 
        not contain $(x_{11},x_{22},x_{33},x_{44})$ we may assume 
	$x_{11}\not\in P$. Inverting $x_{11}$ and using the standard 
        localization trick for determinantal ideals one sees that 
        $PS_{x_{11}}$ contains, up to a change of variables, $I_d(Y_{d-1})$.  
	Hence $\height\,P=\height\, PS_x\geq 4$. 
        Now that we know that $J$ has height $4$ to prove that $J$ is not 
	prime for $d\geq 4$ it is enough to observe that 
        $J\subset (x_{11},x_{22},x_{33},x_{44})$. The latter is  
        straightforward since mod $(x_{11},x_{22},x_{33},x_{44})$ 
        the submatrix of $Y$ consisting of the first $4$-rows has rank $2$. 
  \end{itemize}
\end{proof}

\section{Obstructions to algebraic properties and asymptotic behavior}
\label{sec:obstruction}

In this section we return to the study of LSS-ideals $L_G^\KK(d)$. Using results from 
	\ref{sec:stabilization} and results about $I_{d+1}(X^\generic_{B_d})$ from \ref{sec:minorgraph} we derive necessary conditions for
	$L_G^\KK(G)$ to be a complete intersections or prime. In addition, we discuss the exact asymptotic behavior of these properties
	for complete and complete bipartite graphs. 
	To this end it is convenient to introduce the following notation.  Given an algebraic property $\pP$ of ideals and a 
	graph $G$ we set  

	$$\asym_\KK(\mathcal P,G)=\inf\{ d~ :~ L_G^\KK(d') \mbox{ has property } \mathcal P  \mbox{ for all } d'\geq d\}.$$

	Here we  interested in the properties 
	$\pP \in \{ \rad,\ci,\prim\}$.   
	By \ref{thm:prime_and_ci}m \ref{cor:prime_and_ci} and \ref{thm:pmd} we know that for every graph $G$ we have
	\begin{eqnarray*}
	 \asym_\KK(\ci  ,G)=\min\{ d :    L_G^\KK(d) \mbox{ is } \ci \}\leq \pmd(G) \\
	 \asym_\KK(\prim  ,G)=\min\{ d :    L_G^\KK(d) \mbox{ is } \prim \}\leq \pmd(G)+1 \\ 
	 \asym_\KK(\ci  ,G)\leq \asym_\KK(\prim  ,G)\leq  \asym_\KK(\ci  ,G)+1
	\end{eqnarray*} 

	Furthermore there are graphs such that  $\asym_\KK(\prim  ,G)= \asym_\KK(\ci  ,G)+1$ (e.g.  odd cycles or forests) 
	and others such that $\asym_\KK(\prim  ,G)= \asym_\KK(\ci  ,G)$  (e.g. even cycles). 
	We have the following obstructions: 

	\begin{Proposition}
	  \label{obstructions1}
	  Let $G=([n], E)$. Then: 
	  \begin{itemize}
	    \item[(1)] If $L_G^\KK(d)$ is prime then $G$ does not contain   $K_{a,b}$ with $a+b=d+1$. 
	      Furthermore, if $d>3$ and $\chara \KK=0$ then $G$ does not contain $B_d$. 
	    \item[(2)] If $L_G^\KK(d)$ is a complete intersection then $G$ does not 
	      contain  $K_{a,b}$ with $a+b=d+2$. 
	      Furthermore, if $d>2$ and 
	      $\chara \KK=0$  then $G$ does not contain $B_{d+1}$.
	  \end{itemize}
	\end{Proposition}

	\begin{proof} 
	 \begin{itemize}
	   \item[(1)] The first assertion has been already proved in \ref{obsKab}. For the second let 
	      $\chara \KK = 0$ and $d > 3$. By contradiction, assume $G$ contains $B_d$. Then by \ref{cor:prime_and_ci}  
	      we know that $L_{B_d}^\KK(d)$ is prime because $L_G^\KK(d)$ is prime.     
	      Then \ref{prop:traG} implies that $I_{d+1}(X^\generic_{B_d})$ is prime for a  
	      generic matrix $X$ of arbitrary size and this contradicts  
	      \ref{lem:minorKab}(2).

	   \item[(2)] Assertion (2) follows from (1) by using \ref{thm:prime_and_ci}. 
	 \end{itemize}
	 \end{proof}

	 Another obstruction is described in the following proposition.

	\begin{Proposition}
	  \label{knprci} 
	  Let $\KK$ be a field of characteristic $0$ and $n\in \NN$. 
	  Let $w_n$ be  the largest positive integer such that ${w_n \choose 2} 
	  \leq n$. Then: 
	  \begin{itemize}
	    \item[(1)]  $L^\KK_{K_n}(d)$ is not prime for $d=n+{w_n-2 \choose 2}-1$. 
	    \item[(2)]  $L^\KK_{K_n}(d)$ is not a complete intersection for    
		      $d=n+{w_{n+1}-2 \choose 2}-2 $. 
	  \end{itemize}
	\end{Proposition}

	\begin{proof} 
	\begin{itemize}
	  \item[(1)] We set $h_n={w_n \choose 2}$ and $m_n = w_n+d-1$. 
	    The numbers are chosen so that, using the formulas for the height of determinantal ideals mentioned 
	    in \ref{sec:minorgraph}, the ideal $I_{d+1}(X)$ of $(d+1)$-minors 
	    of a generic symmetric $m_n \times m_n$ matrix $X$ has height $h_n$. 
	    Consider $K_n$ as the graph $([m_n],{[n] \choose 2})$ 
	    on $m_n$ vertices where the vertices $n+1,\ldots, m_n$ do not 
	    appear in edges. Assume, by contradiction, that the ideal $L^\KK_{K_n}(d)$ is 
	    prime. Then by \ref{traS} the ideal  $I_{d+1}^\KK(X_{K_n}^\symmetric)$ 
	    is prime and of height $h_n$. But one has 
	 
	    \begin{eqnarray} 
	      \label{eq:test}
	      I_{d+1}^\KK(X_{K_n}^\symmetric) \subset (x_{11},x_{22},\ldots, x_{h_nh_n})  
	    \end{eqnarray} 	

	    which is a contradiction. To check \eqref{eq:test} it is 
	    enough to prove that the rank of the matrix 
	    $$X_{K_n}^\symmetric \mod (x_{11},x_{22}, \ldots , x_{h_nh_n})$$   
	    is at most $d$. That is, we have to check that  the rank of an 
	    $(m_n \times m_n)$-matrix 
	    with block decomposition 
	    $$
	      \left(
	      \begin{array}{cc}
		0 & A \\
		B & C 
	      \end{array}
	      \right)
	    $$

	    where $0$ is the zero matrix of size $(h_n \times n)$, is at most $d$. 
	    Since $d=m_n-n+m_n-h_n$ the latter is obvious.  

	  \item[(b)] We set $h_n={w_{n+1} \choose 2}$ and $m_n = w_{n+1}+d-1$. 
	    As above, the numbers are chosen so that  the ideal 
	    $I_{d+1}(X)$ of $(d+1)$-minors of a generic 
	    symmetric $m_n \times m_n$ matrix $X$ has height $h_n$. 
		
	    Assume, by contradiction, that  $L_{K_n}^\KK(d)$ is a complete intersection.
	    From \ref{traS} it follows that 
	    $I_{d+1}^\KK(X_{K_n}^\symmetric)$  has height $h_n$.
	    But 
	    \begin{eqnarray} 
	      \label{eq:test2}
	      I_{d+1}^\KK(X_{K_n}^\symmetric) \subset 
		      (x_{11},x_{22},\ldots, x_{h_n-1,h_n-1})  
	    \end{eqnarray} 
	    which is a contradiction. As above \eqref{eq:test2} boils down to an 
	    obvious statement about the rank of a matrix with a zero submatrix of 
	    a certain size. 
	  \end{itemize}
	\end{proof}

	Using this result we can now analyze the asymptotic behavior of
	both $\asym_\KK(\ci,K_n)$ and $\asym_\KK(\prim,K_n)$.

	\begin{Corollary}
	  \label{cor:kn}
	  Let $\KK$ be a field of characteristic $0$. Then
		\begin{eqnarray}
		\label{eq:lim}		 
	 \lim_{n \rightarrow \infty} \frac{\asym_\KK(\ci,K_n)}{n} 
	     =  \lim_{n \rightarrow \infty} \frac{\asym_\KK(\prim,K_n)}{n} 
	     = 2.
		\end{eqnarray}
	\end{Corollary}
	\begin{proof}
	  By \ref{cor:bounds} we have $\asym_\KK(\prim, K_n)\leq 2n-2$. By \ref{knprci}  we have 
	  \begin{eqnarray}
	    \label{eq:primbound}
		  n+{w_{n+1}-2 \choose 2} -1 & \leq & \asym_\KK(\ci,K_n) \leq 
		       \asym_{\KK}(\prim,K_n) 
	  \end{eqnarray}
	  Hence the equalities in \eqref{eq:lim} follow from the fact that 
	  $$\lim_{n \rightarrow \infty} \frac{{w_{n+1} -2\choose 2}}{n} =1.$$
	\end{proof}

	Using \ref{knprci} and   \ref{thm:prime_and_ci} we  obtain further obstructions.

	\begin{Corollary}
	  \label{cor:obstruction}
		Let $G$ be a graph on $n$ vertices and $\KK$ a field of characteristic $0$ and denote by $\alpha = \omega(G)$ the clique number of $G$. 
		Then  $L_G^\KK(d)$ is not a complete intersection
		for  $d\leq \alpha+{w_{\alpha+1}-2\choose 2}-2$ and $L_G^\KK(d)$ is not prime for $d\leq \alpha+{w_{\alpha} -2\choose 2}-1$ where $w_{\alpha}$ is  defined as in \ref{knprci}.
	\end{Corollary}
	 
	 To get an actual feeling of the obstruction, we just explicit one example: 
	 
	 \begin{Example} For $n=15$  one has $w_n=6$ and $L_{K_n}^\KK(d)$ is not prime for $d=15+{ 6-2 \choose 2}-1=20$. Therefore $L_{G}^\KK(20)$ is not prime if $G$ contains $K_{15}$, i.e. 
	 $\omega(G)\geq 15$. 
	 \end{Example}

	For the case of complete bipartite graphs $K_{m,n}$ results of 
	De Concini and Strickland \cite{DS} or Musili and Seshadri \cite{MS} on the varieties of complexes 
	imply  the following:  

	\begin{Theorem} 
	  \label{thm:DS}
	  Let $G =K_{m,n}$. Then: 
	  \begin{itemize}
	  \item[(1)] $L_G^\KK(d)$ is radical for every $d$. 
	  \item[(2)] $L_G^\KK(d)$ is a complete intersection if and only if  
		  $d\geq m+n-1$. 
	  \item[(3)] $L_G^\KK(d)$ is  prime if and only if $d\geq m+n$. 
	 \item[(4)] $\pmd(G)=m+n-1$. 
	  \end{itemize} 
	\end{Theorem} 

	\begin{proof} 
	  Taking into account \ref{rem:LSSasDet}, the assertions (1), (2), and (3) 
	  follow form general results on the variety of complexes proved from 
	  \cite{DS} and, with different techniques, from \cite{MS}. 
	  It has been observed by Tchernev \cite{T} that the assertions in 
	  \cite{DS} that refer to the so-called Hodge algebra structure of the 
	  variety of complexes in \cite{DS} are not correct. However, those 
	  assertions can be replaced with statements concerning Gr\"obner bases  
	  as it is done, for example, in a similar case in \cite{T}.
	  Hence, (1),(2) and (3) can still be deduced from the arguments in \cite{DS}.  
	 
	  Alternative proofs of (2) and (3) are obtained combining \ref{obstructions1} and  \ref{cor:bounds}.  Finally  (4) is a consequence of  \ref{pmd} and \ref{obstructions1}.
	\end{proof}

	\section{Questions and open problems}
	\label{sec:questions}

	We have seen that for the properties
	 ``complete intersection"  and ``prime" of $L_G^\KK(d)$ there is persistence along the parameter $d$  but \ref{ex:nrad} shows persistence does not occur in general  for the property of being radical.

	\begin{Question}
	  \label{q:radical}	 
	  What patterns can occur in the set $\{ d  : L_G^\KK(d) \mbox{ is radical}\}$
	  for a graph $G$?  
	\end{Question}

	Since the complete intersection property and  prime property of $L_G^\KK(d)$ for a given $d$ are inherited by subgraphs, the properties can be characterized by means of forbidden subgraphs. We have explicitly identified the forbidden subgraphs  in  \ref{thm:prim3} for $d=2$ and complete intersection and for $d=3$ and prime. 
	For $d=3$ and complete intersection we do not even have a conjecture for the set of forbidden graphs.  
	For $d = 4$  results from Lov\'asz's book \cite[Ch 9.4]{L} suggest the following: 

	\begin{Question} 
	  Is it true that $L_G^\KK(4)$ is prime if and only if $G$ does not 
	  contain  $K_{a,b}$ for $a+b = 5$ and $B_4$?
	\end{Question} 

	Via the fact that primeness of $L_G^\KK(d)$ implies primeness of
	$I_{d+1}^\KK(X_G)$ a result by Giusti and Merle \cite[Thm. 1.6.1]{GM} guides the intuition 
	behind the following question.

	\begin{Question}
	  Let $G$ be a subgraph of $K_{m,n}$ graph and assume $m \leq n$. 
	  Is it true that $L_G^\KK(m-1)$ is prime if and only if $G$ does not
	  contain  $K_{a,b}$ for $a+b \geq m$? 
	\end{Question}

	By \ref{prop:traG} and \ref{traS} we know that if $L_G^\KK(d)$ is radical or 
	prime then so are $I_{d+1}^\KK(X_G^\generic)$ and $I_{d+1}^\KK(X_G^\symmetric)$ 
	respectively. But our general bounds for $\asym_\KK(\rad,G)$ and
	$\asym_\KK(\prim,G)$ from \ref{cor:bounds} are not good enough to 
	make use of this implication.
	Indeed, \ref{cor:kn} shows that for the properties complete intersection and
	prime and $n$ large enough there are graphs $G$ for which \ref{traS} does not
	prove primality of an interesting ideal. 
	On the other hand the use of \ref{thm:forest} in \ref{pro} shows that 
	one can take advantage of this connection in some cases.
	It would be interesting to exhibit classes different from forests where this
	is possible.

	\begin{Question}
	  \label{qu:classesGraph}
	  Are there more interesting classes of graphs $G = ([n],E)$ for
	  which $\asym_\KK(\ci,G) \leq n-1$ or $\asym_\KK(\prim,G) \leq n$ ?
	\end{Question}

	Despite the fact that \ref{knprci} destroys the hope for using 
	\ref{pro} for general graphs, it would be interesting replace the
	asymptotic result by an actual value.
	By \ref{cor:kn} for $n$ large we have $\asym_\KK(\prim,K_n) = 2n-c_n$
	for some numbers $c_n \in o(n)$ which using the notation of
	\ref{knprci} satisfy $n-{w_n -2\choose 2} +1 \geq c_n \geq 2$.
	But we have no conjecture for an 
	actual formula for $c_n$. 

	\begin{Question}
	  What is the exact value of $\asym_\KK(\prim,K_n)$?
	\end{Question}

	For radicality we have a concrete conjecture in the case $G = K_n$. 

	\begin{Conjecture} 
	 \label{KN} 
	  $$\asym_\KK(\rad,K_n)=1 \mbox{ ~ ~ ~ (at least if } \chara \KK = 0 \mbox{\,)}.$$ 
	  In other words, given a matrix of variables $X$ of size $n\times d$ we conjecture
	  the ideal of the off-diagonal entries of $XX^T$  is radical for all $n,d$. 
	 \end{Conjecture} 

	It would also be interesting to study the ideal generated by all the entries 
	of $XX^T$. We note that the symplectic version of this problem has been 
	investigated by De Concini in \cite{D}. 

	Next we turn to open problems about hypergraph LSS-ideals.
	We know from \ref{thm:pmd} that for a hypergraph $H = (V,E)$ for which 
	$E$ is a clutter the ideal $L_H^\KK(d)$ is
	a radical complete intersection for $d \geq \pmd(G)$. 
	But we prove in \ref{thm:pmd} that $L_H^\KK(d)$ is prime for 
	$d \geq \pmd(H)+1$ only in the case that $H$ is a graph.

	\begin{Question}
	  \label{qu:tensor}
	  Is it true that for a hypergraph $H = (V,E)$, where $E$ is 
	  a clutter, we have $L_H^\KK(d)$ is prime for 
	  $d \geq \pmd(H)+1$?
	\end{Question}

	Similarly, the persistence results from \ref{thm:prime_and_ci}
	ask for generalizations.

	\begin{Question}
	  \label{qu:hyper_prime_and_ci}
	  Let $H = (V,E)$ be a hypergraph, where $E$ is 
	  a clutter. Is it true that
	  if $L_H^\KK(d)$ is a complete intersection (resp.~prime)
	  then so is $L_H^\KK(d+1)$? 
	\end{Question}

	For a number $r \geq 1$ we call a hypergraph $H = (V,E)$ an $r$-uniform graph 
	if every element of $E$ has cardinality $r$. In particular, $E$ is a
	clutter.  We say that an $r$-uniform graph $H = (V,E)$ 
	is $r$-partite if there is a partition $V = V_1 \cup \cdots \cup V_r$ such that
	$\# (A \cap V_i) = 1$ for all $i \in [r]$ and for all $A\in E$.   
	Now we connect the study of ideal $L_H^\KK(d)$ for $r$-uniform ($r$-partite)
	graphs with the study of coordinate sections of the variety of tensors with
	a given rank.  We consider two mappings:

	\begin{itemize}
	  \item[($\phi$)] 
	       Let $\ee_1,\ldots, \ee_n$ be the standard basis vectors of $\KK^n$. 
	       For vectors $v_i = (v_{i,j})_{j \in[d]}
	      \in \KK^d$, $i \in [r]$, consider the map $\phi$ that sends
	       $(v_1,\ldots, v_r) \in (\KK^d)^n$ to the tensor 
	       $$\sum_{j=1}^d \sum_{\sigma \in S_r} 
	       v_{\sigma(1),j} \cdots v_{\sigma(r),j} \, \ee_{\sigma(1)} \otimes \cdots 
	       \otimes \ee_{\sigma(r)} \in \underbrace{\KK^n \otimes \cdots \otimes \KK^n}_r.$$ 
	       We take the sums over the different  
	       tensors arising from 
	       $\ee_{i_1} \otimes \cdots \otimes \ee_{i_r}$, for numbers 
	       $1 \leq i_1 \leq \cdots \leq i_r \leq n$, 
	       by permuting the positions as standard basis 
	       of the space of symmetric tensors.
	    
	  \item[($\psi$)] Let $n = n_1 + \cdots + n_r$ for natural numbers
	       $n_1,\ldots, n_r \geq 1$. Let $\ee_{i}^{(j)} \in \KK^{n_j}$ be
	       the $i$-th standard basis vector of $\KK^{n_j}$, $i \in [n_j]$ and
	       $j \in [r]$. For vectors $v_i^{(j)} = (v_{i,j})_{j \in[d]}
	       \in \KK^{d}$ for $i \in [n_j]$ and $j \in [r]$ consider the map
	       $\psi$ that sends $(v_i^{(j)})_{(i,j) \in [n_j] \times [r]}$ to
	       $$\displaystyle{\sum_{(i_1,\ldots ,i_r) \in [n_1] \times \cdots \times [n_r]}} 
	       v_{i_1}^{(1)} \cdots v_{i_r}^{(r)} \, e_{i_1}^{(1)} \otimes \cdots \otimes \ee_{i_r}^{(r)} \in 
	       \KK^{n_1} \otimes \cdots \otimes \KK^{n_r}.$$
	       We take the tensors $\ee_{i_1}^{(1)} \otimes \cdots \otimes \ee_{i_r}^{(r)}$
	       for numbers $i_j \in [n_j]$, $j \in [r]$ as the standard basis of 
	       $\KK^{n_1} \otimes \cdots \otimes \KK^{n_r}$.
	\end{itemize}

	Recall that a (symmetric) tensor has (symmetric) rank $\leq d$ it can be written 
	as a sum of $\leq d$ decomposable (symmetric) tensors. 
	For more details on tensor rank and the geometry of bounded rank tensors we
	refer the reader to \cite{La}.
	Let $H = (V,E)$ be a hypergraph. We write $\vV(L_H^\KK(d))$ for the vanishing 
	locus of $L_H^\KK(d)$. The definition of the maps $\phi$ and $\psi$ 
	immediately implies the following proposition.
	 
	\begin{Proposition}
	  \label{prop:tensor}
	  Let $H = ([n],E)$ be an $r$-uniform hypergraph and $\KK$ an algebraically 
	  closed field. 
	  \begin{itemize}
	     \item[(1)] Then the restriction of the map 
	       $\phi$ to $\vV(L_H^\KK(d))$ is a parametrization of the 
	       variety  of symmetric tensors in $\underbrace{\KK^{n} \otimes 
	       \cdots \otimes \KK^n}_{r}$ of rank $\leq d$ which when 
	       expanded in the standard basis has zero coefficient for the
	       basis elements indexed by $1 \leq i_1 < \cdots < i_r \leq n$ 
	       and $\{i_1,\ldots, i_r \} \in E$.  
	       In particular, the Zariski-closure of the image of the restriction
	       is irreducible if $L_H^\KK(d)$ is prime.
	     \item[(2)] If $H$ is $r$-partite with respect to the partition 
	       $V = V_1 \cup \cdots \cup V_r$ where $|V_i| = n_i$, $i \in[r]$. 
	       Then the restriction of the map 
	       $\psi$ to $\vV(L_H^\KK(d))$ is a parametrization of the 
	       variety  of tensors in $\KK^{n_1} \otimes 
	       \cdots \otimes \KK^{n_r}$ of rank $\leq d$ which when expanded in 
	       the standard basis have zero coefficient for the basis elements
	       indexed by $i_1 , \ldots, i_r$ where $\{i_1,\ldots, i_r \} \in E$. 
	       In particular, the Zariski-closure of the image of the 
	       restriction is irreducible if $L_H^\KK(d)$ is prime.
	  \end{itemize}
	\end{Proposition}

	\ref{prop:tensor} gives further motivation to \ref{qu:tensor}. Indeed,
	it suggests to strengthen \ref{qu:classesGraph}.

	\begin{Question}
	    \label{qu:classesHypergraph}
	    Let $\KK$ be an algebraically closed field.
	    Can one describe classes of $r$-uniform hypergraphs $H$ for which
		$L_H^\KK(d)$ is
	    prime for some $d$ bounded from above
	    by the
		maximal symmetric rank of a symmetric sensor in
	       $\underbrace{\KK^{n} \otimes 
	       \cdots \otimes \KK^n}_{r}$.
	\end{Question}

	An analogous question can be asked for $r$-partite $r$-uniform hypergraphs and tensors of
	bounded rank.

	\end{document}